\documentclass[11pt, reqno]{amsart}
\usepackage[dvipsnames]{xcolor}
\usepackage{mathtools}
\usepackage[english]{babel}
\usepackage{etoolbox}
\usepackage{amsthm, graphicx}
\usepackage[expansion=false]{microtype}
\usepackage[export]{adjustbox}
\usepackage{blindtext}
\usepackage{latexsym}
\usepackage[pagewise]{lineno}
\usepackage[utf8]{inputenc}
\usepackage{exscale}
\usepackage{amsmath}
\usepackage{amssymb}
\usepackage{amsfonts}
\usepackage{mathrsfs}
\usepackage{amsbsy}
\usepackage{relsize}
\usepackage{comment}
\PassOptionsToPackage{reqno}{amsmath}
\usepackage{esint}
\usepackage{stmaryrd}
\usepackage{cite} 
\usepackage{tikz}

\usepackage{overpic}
\usepackage[T2A,T1]{fontenc}

\usepackage[margin=1in]{geometry}
\usepackage{microtype}

\numberwithin{equation}{section}
\usepackage[scaled]{helvet}
\DeclareMathOperator{\sgn}{sgn}
\newcommand{\bb}{\mathbf{B}}

\newcommand{\norm}[1]{\lVert #1 \rVert}
\newcommand{\LL}{\mathscr{L}}
\newcommand{\bx}{\mathbf{X}}

\newcommand{\R}{\mathbb{R}}

\newcommand{\E}{\mathbb{E}}

\newcommand{\bbs}{\mathbb{S}}

\newcommand{\fq}{\mathfrak{q}}

\definecolor{Cgrey}{rgb}{0.85,0.85,0.85}
\definecolor{Cblue}{rgb}{0.50,0.85,0.85}
\definecolor{Cred}{rgb}{1,.2,.4}
\definecolor{fancy}{rgb}{0.10,0.85,0.10}
\definecolor{mygreen}{rgb}{0.01,0.6,0.2}
\definecolor{tealgreen}{rgb}{0.0, 0.51, 0.5}
\definecolor{tangerine}{rgb}{0.95, 0.52, 0.0}
\definecolor{saffron}{rgb}{0.96, 0.77, 0.19}
\definecolor{mint}{rgb}{0.24, 0.71, 0.54}
\definecolor{lincolngreen}{rgb}{0.11, 0.35, 0.02}
\definecolor{lava}{rgb}{0.81, 0.06, 0.13}
\definecolor{lasallegreen}{rgb}{0.03, 0.47, 0.19}
\definecolor{mahogany}{rgb}{0.75, 0.25, 0.0}
\definecolor{electricultramarine}{rgb}{0.25, 0.0, 1.0}
\definecolor{mypink1}{rgb}{0.858, 0.188, 0.478}
\definecolor{mypink2}{RGB}{219, 48, 122}
\definecolor{mypink3}{cmyk}{0, 0.7808, 0.4429, 0.1412}
\definecolor{mygray}{gray}{0.6}
\definecolor{venetianred}{rgb}{0.78, 0.03, 0.08}
\definecolor{sapphire}{rgb}{0.03, 0.15, 0.4}
\definecolor{utahcrimson}{rgb}{0.83, 0.0, 0.25}
\definecolor{trueblue}{rgb}{0.0, 0.45, 0.81}
\definecolor{carminered}{rgb}{1.0, 0.0, 0.22}
\definecolor{cobalt}{rgb}{0.0, 0.28, 0.67}
\definecolor{cornflowerblue}{rgb}{0.39, 0.58, 0.93}
\definecolor{tangerineyellow}{rgb}{1.0, 0.8, 0.0}
\definecolor{mypink1}{rgb}{0.858, 0.188, 0.478}
\definecolor{mypink2}{RGB}{219, 48, 122}
\definecolor{mypink3}{cmyk}{0, 0.7808, 0.4429, 0.1412}
\definecolor{mygray}{gray}{0.6}
\definecolor{venetianred}{rgb}{0.78, 0.03, 0.08}
\definecolor{sapphire}{rgb}{0.03, 0.15, 0.4}
\definecolor{utahcrimson}{rgb}{0.83, 0.0, 0.25}
\definecolor{trueblue}{rgb}{0.0, 0.45, 0.81}
\definecolor{carminered}{rgb}{1.0, 0.0, 0.22}
\definecolor{cobalt}{rgb}{0.0, 0.28, 0.67}
\definecolor{cornflowerblue}{rgb}{0.39, 0.58, 0.93}
\definecolor{darkmagenta}{rgb}{0.55, 0.0, 0.55}

\definecolor{bluegreen}{rgb}{0.0, 0.3, 0.9}
\definecolor{britishracinggreen}{rgb}{0.0, 0.26, 0.15}
\definecolor{cornellred}{rgb}{0.7, 0.11, 0.11}
\definecolor{darkpastelgreen}{rgb}{0.01, 0.75, 0.24}
\definecolor{emerald}{rgb}{0.31, 0.78, 0.47}
\definecolor{ferrarired}{rgb}{1.0, 0.11, 0.0}
\definecolor{DarkBlue}{rgb}{0.0, 0.0, 0.55}
\definecolor{DarkRed}{rgb}{0.55, 0.0, 0.0}
\usepackage[
  colorlinks=true,
  linkcolor=DarkRed,
  citecolor=DarkBlue,
  urlcolor=Cblue,
]{hyperref}

\newtheorem{defi}{Definition}[section]
\newtheorem{prop}{Proposition}[section]
\newtheorem{thm}{Theorem}[section]
\newtheorem{rem}{Remark}[section]
\newtheorem{lem}{Lemma}[section]

\renewcommand{\theequation}{\thesection.\arabic{equation}}
\numberwithin{equation}{section}

\usepackage{lipsum}
\usepackage{tikz,xcolor}

\definecolor{lime}{HTML}{A6CE39}
\DeclareRobustCommand{\orcidicon}{
\begin{tikzpicture}
\draw[lime, fill=lime] (0,0) 
circle [radius=0.16] 
node[white] {{\sffamily \tiny ID}};
\draw[white, fill=white] (-0.0625,0.095) 
circle [radius=0.007];
\end{tikzpicture}
\hspace{-2mm}
}
\foreach \x in {A, B}{\expandafter\xdef\csname orcid\x\endcsname{\noexpand\href{https://orcid.org/\csname orcidauthor\x\endcsname}
{\noexpand\orcidicon}}
}

\usepackage{bm}

\title[Fujita exponent]{The Fujita exponent across an interface}

\author[M. Majdoub, E. Mliki]{Mohamed Majdoub\orcidA{} \& Ezzedine Mliki\orcidB{}}
\address{Department of Mathematics, College of Science, Imam Abdulrahman Bin Faisal University, P. O. Box 1982, Dammam, Saudi Arabia}
\address{Basic and Applied Scientific Research Center, Imam Abdulrahman Bin Faisal University, P.O. Box 1982, 31441, Dammam, Saudi Arabia}
\email{\tt mmajdoub@iau.edu.sa}
\email{\tt med.majdoub@gmail.com}
\email{\tt mohamed.majdoub@fst.rnu.tn}
\email{\tt ermliki@iau.edu.sa}

\begin{document}
\begin{abstract}
We consider the semilinear parabolic equation
\[
\partial_t u = \Delta u + 2\mathfrak{q}\,\delta_{\bbs}\,\nabla u + |u|^{p-1}u
\qquad \text{in } (0,\infty)\times\mathbb{R}^N,
\]
where $|\fq|\le 1$, $p>1$, and $\bbs$ is a fixed interface hyperplane.

Working in Lebesgue spaces, we first establish local well-posedness of mild solutions. This is achieved by combining Gaussian bounds for the associated fundamental solution with a contraction mapping argument adapted to the lack of spatial homogeneity induced by the interface term.

We then prove a sharp Fujita-type dichotomy for nonnegative solutions. Specifically, we show that every nontrivial solution blows up in finite time when $1<p \le 1+\frac{2}{N}$, whereas for $p>1+\frac{2}{N}$ global solutions exist for sufficiently small initial data. The blow-up analysis relies on a suitably adapted test-function method that accounts for the presence of the interface.

It is noteworthy that the critical exponent coincides with the classical Fujita exponent for the heat equation, indicating that the Fujita phenomenon remains stable under the presence of discontinuous diffusion effects and interface transmission conditions. To the best of our knowledge, this is the first result of this type for operators involving a singular drift supported on a hypersurface.
\end{abstract}
\subjclass[2020]{35K58, 35K15, 35B44, 35B33, 35R05, 60J60, 60J55}

\keywords{
Fujita critical exponent; Fujita phenomenon; finite-time blow-up; global existence; semilinear parabolic equations; transmission conditions; singular interface drift; skew Brownian motion; local time
}
\date{\today}
\maketitle
%\tableofcontents

\renewcommand{\theequation}{\thesection.\arabic{equation}}
%==========================================================================
\section{Introduction}
\label{sec:intro}
%==========================================================================

We study the Cauchy problem for the semilinear parabolic equation
\begin{equation}
\begin{cases}
\partial_t u(t,x) = \mathscr{L}u(t,x) + |u(t,x)|^{p-1}u(t,x),  & \text{in } (0, \infty)\times\mathbb{R}^N,\\
u(0,x) =u_0(x), & \text{in } \mathbb{R}^N,
\end{cases}
\label{Main-eq}
\end{equation}
where $p>1$, $N\ge 2$, and $\mathscr{L}$ is a second-order
operator whose drift is a singular measure concentrated on a fixed hyperplane~$\bbs$ given by
\begin{equation}\label{L-operator}
\mathscr{L} = \Delta + 2\mathfrak{q}\, \delta_{\bbs}\, \nabla.
\end{equation}
Here $\Delta$ is the Laplacian on $\R^N$, the hyperplane
$\bbs=\{(x_1,\ldots,x_N)\in\R^N\mid x_N=0\}$ serves as a fixed interface,
$\fq$ is a real parameter with $|\fq|\le 1$, and $\delta_{\bbs}$ is the surface
distribution defined by
\[
\int_{\mathbb{R}^N} \varphi(x)\,\delta_{\bbs}(x)\,dx
=\int_{\bbs}\varphi(x)\,d\sigma(x),
\qquad \varphi\in C_c^\infty(\R^N).
\]

The operator $\mathscr{L}$ acts as the ordinary Laplacian in each of the
half-spaces $\{x_N>0\}$ and $\{x_N<0\}$, but imposes a singular drift
concentrated on~$\bbs$ that couples the two sides through a transmission
condition on the normal derivative.  Concretely, a smooth function $f$ belongs
to the domain of $\mathscr{L}$ if and only if it satisfies
\begin{equation}\label{eq:transmission-intro}
(1-\fq)\,\partial_{x_N}f(\tilde x,0^+)
\;=\;
(1+\fq)\,\partial_{x_N}f(\tilde x,0^-),
\qquad \tilde x\in\R^{N-1},
\end{equation}
so the parameter $\fq$ controls the imbalance of diffusive flux across~$\bbs$.
When $\fq=0$ the transmission condition reduces to continuity of the normal
derivative and $\mathscr{L}$ coincides with $\Delta$ on all of~$\R^N$.  For
$\fq\neq 0$ the interface acts as a partially permeable, biased barrier:
the diffusion is enhanced on one side and impeded on the other, while the
tangential dynamics remain unaffected.

\medskip

This class of operators was introduced and developed by Portenko~\cite{Portenko1, Portenko2, Portenko3}, who constructed diffusion processes whose generators include a generalized drift supported on a hypersurface. The fundamental solution of the linear equation $(\partial_t - \mathscr{L})u = 0$ was obtained explicitly by Mastrangelo and Talbi~\cite{Talbi}. A comprehensive account of its analytic and probabilistic properties is provided in the survey by Lejay~\cite{Lejay2006}; see also~\cite{Zili-2016}.

The work~\cite{Talbi} is motivated, at its origin, by Portenko’s studies~\cite{Portenko1, Portenko2} on parabolic equations with globally regular diffusion coefficients and generalized drifts supported on a surface. It is also inspired by ~\cite{Okada} who treats the dimension one, where diffusion coefficients are piecewise regular and the drift is supported at discontinuity points of the coefficient. 

In higher dimensions, the situation becomes significantly more intricate. Unlike problems involving reflected processes with smooth coefficients, the multi-dimensional setting cannot, in general, be reduced to the one-dimensional case. As observed in~\cite{Talbi}, even in the simplest situation where the discontinuity set is a hyperplane and the coefficients are piecewise constant, the transition probabilities do not factorize. This non-factorization arises from the fact that the transition probabilities of the components tangent to the hyperplane depend on the diffusion coefficient, which itself varies with the normal direction.

The analysis in~\cite{Talbi} is carried out in a relatively simple but still relevant framework in which the diffusion matrix is diagonal. In this case, changes in the diffusion coefficient can be interpreted as local changes of the time scale. This observation naturally leads to the notion of a modified skew Brownian motion associated with a diffusion coefficient that is piecewise regular.

The operator $\mathscr{L}$ admits a transparent probabilistic interpretation
through the theory of Feller semigroups and skew Brownian motion.  Define the
family of operators $\{\mathbf{T}(t)\}_{t\ge 0}$ by
\[
(\mathbf{T}(t)\varphi)(x)
\;=\;
\int_{\R^N}G_{\fq}(x,y,t)\,\varphi(y)\,dy,
\qquad t>0,
\]
with $\mathbf{T}(0)=I$, where $G_{\fq}(x,y,t)$ is the fundamental solution
of~$(\partial_t-\mathscr{L})u=0$ as stated in Theorem~\ref{Fundamental} below.  Positivity of $G_{\fq}$, conservation of
mass ($\int G_{\fq}(\cdot,y,t)\,dy=1$), and Gaussian upper and lower bounds
(see Proposition~\ref{Cruc-p} below) together imply that $\{\mathbf{T}(t)\}$ is a
conservative, positive, strongly continuous semigroup on $C_0(\R^N)$, hence a
{Markov semigroup}.  By the Hille--Yosida--Feller theory, the diffusion
generated by $\mathscr{L}$ can therefore be realised as a conservative strong
Markov process $X_t=(\widetilde X_t,X_t^N)$ with continuous sample paths.

The associated diffusion process admits the following structure (see, e.g.,~\cite{Rami}). The tangential component
\(\widetilde X_t \in \mathbb{R}^{N-1}\) evolves as a standard \((N-1)\)-dimensional Brownian motion, independent of the interface. The normal component \(X_t^N \in \mathbb{R}\) satisfies the stochastic differential equation
\begin{equation}\label{eq:SDE-intro}
X_t^N
= X_0^N + \sqrt{2}\,B_t + 2\fq\,L_t^0(X^N),
\end{equation}
where \(B\) is a standard one-dimensional Brownian motion and \(L_t^0(X^N)\) denotes the symmetric local time of \(X^N\) at the origin.

Equivalently, \(X^N\) is a one-dimensional skew Brownian motion with skewness parameter
\(\alpha = (1-\fq)/2 \in [0,1]\). Away from the interface, the process behaves as standard Brownian motion; each time the trajectory hits \(\bbs\), it is reflected to the positive side with probability \(\alpha\) and to the negative side with probability \(1-\alpha\).  The parameter~$\fq$ thus quantifies the asymmetry of
crossing: $\fq=0$ (i.e.\ $\alpha=\tfrac12$) yields symmetric Brownian motion,
$\fq>0$ biases trajectories toward $\{x_N<0\}$, and $\fq<0$ biases them
toward $\{x_N>0\}$.

Analytically, this probabilistic structure is reflected in a factorization of
the kernel $G_{\fq}$ into a product of the $(N\!-\!1)$-dimensional Gaussian
heat kernel in the tangential variables and a one-dimensional skew kernel
$g_{\fq}$ in the normal variable.  The skew kernel is itself an explicit
combination of the free heat kernel and a reflected Gaussian correction (see
Theorem~\ref{Fundamental} below).  This factorization is at the heart of all
estimates in the present paper: it explains why $G_{\mathfrak{q}}$ obeys the same two-sided Gaussian bounds as the standard heat kernel, up to multiplicative constants depending on $\mathfrak{q}$, and consequently why the large-scale diffusive behavior of $\mathscr{L}$ is indistinguishable from that of the Laplacian.

The Feller--Markov viewpoint plays two concrete roles in our analysis.  First,
it supplies positivity and mass conservation for the linear flow, which are
essential both for the fixed-point construction of mild solutions and for the
comparison arguments in the blow-up proof.  Second, the equivalence between
the PDE weak formulation, the martingale problem for~$\mathscr{L}$, and the
SDE~\eqref{eq:SDE-intro} (established in
Proposition~\ref{prop:martingale-localtime} below) provides a flexible toolkit for
deriving test-function identities: the cancellation of interface terms in the
Fujita functional, which is the core technical step of Section~\ref{sec:fujita},
is ultimately a consequence of the transmission
condition~\eqref{eq:transmission-intro} and the local-time structure
of skew Brownian motion.

For the classical semilinear heat equation
\begin{equation}\label{eq:intro-heat}
\partial_t u=\Delta u+u^p,
\end{equation}
a fundamental result due to Fujita~\cite{Fujita1966} reveals a sharp
threshold phenomenon governed by the critical exponent
\begin{equation}
    \label{eq:p-f}
    p_F \;=\; 1+\frac{2}{N}.
\end{equation}
When $1<p\le p_F$, every nontrivial nonnegative solution blows up in
finite time, regardless of how small the initial data are. In contrast,
if $p>p_F$, global-in-time solutions exist for sufficiently small initial
data. 

This critical exponent is not accidental: it reflects a precise balance
between the nonlinear amplification $u^p$ and the diffusive smoothing
induced by the heat operator. In particular, the scaling invariance of
\eqref{eq:intro-heat}, together with the Gaussian decay of the heat
kernel and the structure of self-similar solutions, all point to $p_F$
as the unique threshold separating blow-up from global existence. We
refer to Fujita~\cite{Fujita1966, Fuj69}, Hayakawa~\cite{Hayakawa1973},
and Weissler~\cite{Weissler1980} for the classical theory, and to
Fujishima--Kawakami--Sire~\cite{FKS2019} for more recent developments.

The operator $\mathscr{L}$ in~\eqref{L-operator} is neither translation
invariant nor scale homogeneous: the presence of the interface~$\bbs$
breaks spatial homogeneity, while the singular drift
$2\fq\,\delta_{\bbs}\,\nabla$ disrupts the scaling structure that lies at
the heart of the classical Fujita argument. This leads to the central question of this work: does the Fujita threshold~\eqref{eq:p-f} persist for~\eqref{Main-eq}, or is it modified by the interface?

\medskip\noindent
We show that the Fujita exponent is in fact robust. Despite the strong
spatial inhomogeneity encoded in $\mathscr{L}$, the critical threshold
remains precisely $p_F=1+\frac{2}{N}$. The underlying reason is that the
interface affects the constants in the kernel estimates, and therefore in
the corresponding blow-up and existence bounds, without altering the
fundamental scaling exponents. In particular, the two-sided Gaussian
comparison (see Proposition~\ref{Cruc-p}\,(i) below) guarantees that the
large-scale diffusive behaviour of $\mathscr{L}$ coincides with that of
the Laplacian.

We work with two solution concepts, which are equivalent within the appropriate functional framework.
\begin{defi}[{\tt Weak solution}]\label{defn:weak-solution}
Let $N\ge 2$, $p>1$, $|\fq|\leq 1$, and $0<T\leq \infty$.
A function $u$ is a weak solution of \eqref{Main-eq} on
$(0,T)\times\R^N$ if $u_0\in L^1_{\mathrm{loc}}(\R^N)$,
$|u|^{p-1}u\in L^1_{\mathrm{loc}}((0,T)\times\R^N)$, and
\begin{equation}\label{eq:weak-formulation}
\begin{split}
&\int_0^T\!\!\int_{\R^N}
u\,(\partial_t\psi+\Delta\psi)\,dx\,dt
\;+\;
2\fq\int_0^T\!\!\int_{\bbs}
u(t,\tilde x,0)\,\partial_{x_N}\psi(t,\tilde x,0)
\,d\sigma(\tilde x)\,dt \\
&\quad +\;
\int_{\R^N}u_0(x)\,\psi(0,x)\,dx
\;+\;
\int_0^T\!\!\int_{\R^N}|u|^{p-1}u\;\psi\,dx\,dt
\;=\;0
\end{split}
\end{equation}
for all $\psi\in C_c^\infty([0,T)\times\R^N)$.
\end{defi}
Within the appropriate functional framework, \eqref{Main-eq} is equivalent to
the Duhamel (integral) formulation
\begin{equation}\label{Duhamel}
 u(t) = e^{t\LL}u_0 + \int_0^t e^{(t-s)\LL}\bigl(|u(s)|^{p-1}u(s)\bigr)\,ds,
\end{equation}
where $e^{t\LL}$ denotes the linear semigroup generated by $\LL$.  A solution
of~\eqref{Duhamel} is called a {mild solution} of~\eqref{Main-eq}.

Throughout the paper, $\fq$ (fraktur font) denotes the {skewness
parameter} of the interface, satisfying $|\fq|\le 1$.  The unadorned
letter~$q$ is reserved for the {Lebesgue integrability exponent}.
These two quantities play entirely different roles, and the reader should take
care not to confuse them.

Our first result establishes local well-posedness of mild solutions in
Lebesgue spaces.  Beyond existence and uniqueness, the statement provides a
functional framework that explicitly tracks the two competing effects: the
heat-type smoothing away from the interface and the loss of scaling invariance
induced by the transmission structure.
\begin{thm}[{\tt Local Lebesgue well-posedness for \eqref{Main-eq}}]
\label{thm:LWP}
Let $N\ge 2$, $|\fq|\le 1$, $p>1$, and define
\begin{equation}
    \label{eq:q-c}
    q_c=\frac{N(p-1)}{2}.
\end{equation}
Let $1\le q<\infty$ satisfy
\begin{equation}
    \label{eq:ass-q}
    \left(q > q_c \;\;\;\text{and}\;\;\; q\geq 1\right)\quad \text{or}\quad \left(q= q_c \;\;\;\text{and}\;\;\; q> 1\right).
\end{equation}
Then, for every $u_0 \in L^q(\R^N)$, there exist a time $T = T(u_0) > 0$ and a unique mild solution $u \in C([0,T],L^q(\R^N))$ of~\eqref{Main-eq}. In addition, the following properties hold:
\begin{enumerate}
\item[(i)] {(Smoothing effect and continuous dependence on initial data)} For all $t\in (0,T]$,
\begin{equation}
\label{eq:cont-depend}
\|u(t)-v(t)\|_{L^q}
+ t^{\frac{N}{2q}}\|u(t)-v(t)\|_{L^\infty}
\leq C\|u_0-v_0\|_{L^q},
\end{equation}
where $T=\min\{T(u_0),T(v_0)\}$ and the constant $C$ depends only on $\|u_0\|_{L^q}$ and $\|v_0\|_{L^q}$.
\item[(ii)] The solution satisfies
\[
\lim_{t\downarrow 0} t^{\frac{N}{2q}}\|u(t)\|_{L^\infty}=0.
\]
\item[(iii)] {(Positivity preservation)} If $u_0\geq 0$, then $u(t)\geq 0$ for every $t\in [0,T(u_0)]$.
\end{enumerate}
Moreover, the existence time can be chosen uniformly on compact sets: for any compact set $\mathcal{K}\subset L^q(\R^N)$, there exists $T=T(\mathcal{K})>0$ such that the solution of~\eqref{Main-eq} associated with any $u_0\in \mathcal{K}$ exists on $[0,T]$.
\end{thm}

\begin{rem}\label{rem:LWP}
\leavevmode\rm
\begin{enumerate}
\item
Since the kernel of \(e^{t\LL}\) is not given by translations of a fixed profile, the standard convolution-based approach to the existence theory for~\eqref{eq:intro-heat} is no longer applicable. Instead, we establish the contraction property by exploiting pointwise Gaussian bounds on \(G_{\fq}\), which yield \(L^q\)–\(L^r\) smoothing estimates with the same exponents as those of the classical heat semigroup (see Proposition~\ref{prop:semigroup} below).
\item The ``doubly critical'' case \(q=\frac{N(p-1)}{2}=1\) is considerably more delicate and remains open. In particular, it appears that for certain initial data \(u_0 \in L^1\), even local-in-time solutions may fail to exist. Recent results for the classical semilinear heat equation ($\fq=0$) can be found in~\cite{Fuji, Miya} and the references therein.
\end{enumerate}
\end{rem}

Our second and third results together identify the sharp Fujita-type dichotomy for~\eqref{Main-eq}.

\begin{thm}[{\tt Blow-up}]\label{Fujita1}
Let $N\ge2$ and $|\fq|\leq 1$.  If $1<p\le 1+\frac{2}{N}$ and
$u_0\in L^1(\R^N)$ satisfies
\begin{equation}\label{ass-u0}
    \int_{\R^N}u_0(x)\,dx>0,
\end{equation}
then every nonnegative weak solution of~\eqref{Main-eq} (in the sense of
Definition~\ref{defn:weak-solution}) blows up in finite time.
\end{thm}

\begin{thm}[{\tt Global existence}]\label{thm:global-small}
Let $N\ge2$ and $|\mathfrak{q}|\leq 1$.
Assume that $p>1+\frac{2}{N}$ and that the initial data
$u_0\in L^{q_c}(\R^N)$ satisfies
\begin{equation}\label{Ass-GE0}
    \norm{u_0}_{L^{q_c}}\le \varepsilon,
\end{equation}
for some $\varepsilon>0$ sufficiently small.
Then the Cauchy problem~\eqref{Main-eq} admits a global-in-time mild solution.
\end{thm}
\begin{rem}\label{rem:main-results}
\leavevmode\rm
\begin{enumerate}
\item
The conclusion that $p_F=1+\tfrac{2}{N}$ for~\eqref{Main-eq} is intrinsic to the equation and does not arise from a perturbative effect.
The operator $\mathscr{L}$ is neither translation invariant nor homogeneous,
and the interface drift constitutes a genuine distributional term.
The persistence of the Fujita exponent follows from the two-sided Gaussian
comparison (cf.\ Proposition~\ref{Cruc-p}\,(i)), which shows that the interface
affects the constants but leaves the critical scaling unchanged.

\item
The principal distinction between the proof of Theorem~\ref{Fujita1} and the
classical argument for~\eqref{eq:intro-heat} lies in the construction of the
test function. Our choice is adapted to the transmission
condition~\eqref{eq:transmission-intro}, ensuring that no uncontrolled
interface contribution appears in the weak formulation.

\item
The theorem yields a complete dichotomy. If $1<p\le p_F$, there exists no
nontrivial global nonnegative solution. If $p>p_F$, nontrivial global
solutions exist whenever the initial data are sufficiently small.

\item
Although skew Brownian motion provides a natural probabilistic
interpretation of $\mathscr{L}$, the blow-up argument is entirely analytic.
It relies only on positivity and two-sided Gaussian bounds for the
fundamental solution, and therefore extends to any interface operator
admitting comparable kernel estimates.
\end{enumerate}
\end{rem}

The article is organized as follows.
Section~\ref{sec:prelim} reviews the explicit representation of the fundamental solution $G_{\fq}$, establishes the Gaussian bounds and the Chapman--Kolmogorov identity, and presents the equivalence between the weak PDE formulation, the martingale problem associated with~$\mathscr{L}$, and the skew Brownian motion SDE~\eqref{eq:SDE-intro}.  
Section~\ref{sec:lgwp} is devoted to the semilinear problem, where we prove local existence and uniqueness of mild solutions in Lebesgue spaces (Theorem~\ref{thm:LWP}), together with global existence for sufficiently small initial data (Theorem~\ref{thm:global-small}).  
Section~\ref{sec:fujita} contains the Fujita-type analysis and the proof of Theorem~\ref{Fujita1}.  
Finally, Section~\ref{sec:conclusion} concludes the paper with a summary of the main results and a discussion of possible directions for future research.

Throughout the paper, $C$ denotes a positive constant whose value may vary from line to line.\\

%==========================================================================
\section{Preliminaries}
\label{sec:prelim}
%==========================================================================

We begin by fixing notation and collecting the main analytic properties of the
fundamental solution associated with the interface operator $\mathscr L$.

Let $\mathbf{n}=(0,0,\cdots,1)\in\R^N$. Then we can write
$\bbs=\bbs\oplus \bbs^\perp=\bbs\oplus \R\mathbf{n}$.
Every $x\in\R^N$ admits the decomposition
\[
x=\tilde x+\langle x,\mathbf{n}\rangle\mathbf{n}
=(x-\langle x,\mathbf{n}\rangle\mathbf{n})+\langle x,\mathbf{n}\rangle\mathbf{n},
\]
where $\tilde x=(x_1,\ldots,x_{N-1},0)$ denotes the orthogonal projection of $x$
onto $\bbs$.

\medskip
The next theorem recalls the explicit form of the fundamental solution.
Although the formula is known~\cite{Talbi, Zili-2016}, we
state it here because it is the starting point for all kernel bounds and
semigroup estimates used in the paper.

\begin{thm}
    \label{Fundamental}
There exists a unique fundamental solution $G_{\fq}(x,y, t)$ 
to the linear equation $(\partial_t-\LL)u=0$, which can be expressed explicitly  as
\begin{equation}
    \label{Fund-sol}
    \!\!\!G_{\fq}(x,y, t) = \frac{1}{(4\pi t)^{\frac{N}{2}}}\left[ \exp \left( -\frac{|x - y|^2}{4t} \right) -\frac{y_N}{|y_N|} \, \fq \, \exp \left( -\frac{(|x_N| + |y_N|)^2 + |\tilde{x} - \tilde{y}|^2}{4t} \right) \right]
\end{equation}
where $\tilde{x}=(x_1, x_2, \cdots, x_{N-1},0)$ and $\tilde{y}=(y_1, y_2, \cdots, y_{N-1},0)$ denote the orthogonal projections of $x$ and $y$ onto $\bbs$, respectively.
\end{thm}

\begin{rem}\label{rem:fund-sol}
\leavevmode\rm 
\begin{enumerate}
\item The first term is the free heat kernel, while the
second is a reflected Gaussian term weighted by $\mathfrak q$ and $\sgn(y_N)$.
\item The presence of $|x_N|+|y_N|$ (instead of $x_N-y_N$) is the analytic
signature of reflection across $\bbs$.
\item The explicit kernel is one of the main reasons this interface model is a
natural testing ground for Fujita-type blow-up: it is nonhomogeneous, yet still
admits sharp pointwise bounds.
\item As observed in \cite{Talbi}, an analogous representation of the fundamental solution remains valid when the vector $\mathbf{n} = (0,0,\ldots,1)$ is replaced by an arbitrary nonzero vector $\mathbf{n} \in \mathbb{R}^N$. In this more general setting, the hyperplane $\bbs$ is given by
$$
\bbs=\{\, x\in\R^N\,|\, \langle x,\mathbf{n}\rangle=0\,\}.
$$
\end{enumerate}
\end{rem}

\medskip
For the nonlinear problem, the key point is not the explicit formula itself but
rather the structural properties it implies: positivity, mass conservation, and
Gaussian comparisons. These are summarized next.

\begin{prop}
    \label{Cruc-p}
    Let $N\ge 2$ and $|\fq|\le 1$. The following properties hold:
    \begin{itemize}
        \item[\textup{(i)}] \textsf{Two-sided Gaussian comparison:}
        \begin{equation}\label{eq:gauss-compare}
(1-|\mathfrak q|)\,g(x-y,t)\ \le\ G_{\mathfrak q}(x,y,t)\ \le\ (1+|\mathfrak q|)\,g(x-y,t), \quad t>0,\; x,y\in\R^N
\end{equation}
where 
\begin{equation}
    \label{heat-kernel}
    g(z,t)=(4\pi t)^{-N/2}\exp\!\left(-\frac{|z|^2}{4t}\right), \quad t>0,\; z\in\R^N
\end{equation}
denotes the standard heat kernel on $\R^N$.
    \item[\textup{(ii)}] The kernel $G_{\fq}$ is normalized in the $y$--variable:
\begin{equation}
    \label{eq:G-q-y}
    \int_{\R^N} G_{\fq}(x,y,t)\,dy=1.
\end{equation}
    \item[\textup{(iii)}] The integral of $G_{\fq}$ with respect to $x$ is given by
    \begin{equation}
    \label{eq:G-q-x}
    \int_{\R^N} G_{\fq}(x,y,t)\,dx
    =
    1 -\frac{y_N}{|y_N|} \, \fq \, \operatorname{erfc}\left( \frac{|y_N|}{2\sqrt{t}} \right),
    \end{equation}
    where the complementary error function $\operatorname{erfc}$ is given by
    \begin{equation}
        \label{erfc}
        \operatorname{erfc}(z)=\frac{2}{\sqrt{\pi}}\int_{z}^{\infty} e^{-s^{2}}\,ds,
\qquad z\in\mathbb R.
    \end{equation}
    \item[\textup{(iv)}] For every $ 1 \leq r \leq \infty $ and $ t > 0 $,
\begin{equation}
    \label{G-q-Lr}
    \left\|G_{\fq}(\cdot, y, t)\right\|_{L^r} \leq C\, t^{-\frac{N}{2}\left(1 - \frac{1}{r}\right)}, \quad 
    \left\|G_{\fq}(x, \cdot, t)\right\|_{L^r} \leq C\, t^{-\frac{N}{2}\left(1 - \frac{1}{r}\right)},
\end{equation}
where $ C > 0 $ depends only on $ N $ and $\fq $.
    \end{itemize}
\end{prop}

\begin{rem}\label{rem:Cruc-p}
\leavevmode\rm 
\begin{enumerate}
\item Part (i) is the fundamental estimate of the paper: it implies that interface diffusion has the same {diffusive scaling} as the free heat equation, up to constants depending on $\mathfrak q$.
\item Part (ii) expresses the conservation of total mass in the forward variable. This is the analytic form of the conservativeness of the associated Markov process.
\item Part (iii) shows that integration in the other variable is not symmetric.
This asymmetry is expected: the kernel corresponds to a diffusion with a drift
supported on the interface; hence, it is not self-adjoint in general.
\item Part (iv) is the substitute for translation invariance. Even though the
kernel is not a convolution kernel, it satisfies the same $L^r$ scaling as the
Gaussian heat kernel, which is exactly what is needed for semigroup estimates.
\end{enumerate}
\end{rem}

\begin{proof}[Proof of Proposition~\ref{Cruc-p}]
~\begin{itemize}
\item[(i)] 
From \eqref{Fund-sol}, we have
$$
G_{\mathfrak q}(x,y,t)=(4\pi t)^{-N/2}\Bigl(A-\sigma\,\mathfrak q\,B\Bigr),
$$
where
\begin{eqnarray*}
A:&=&\exp\!\left(-\frac{|x-y|^2}{4t}\right),\\
B:&=&\exp\!\left(-\frac{(|x_N|+|y_N|)^2+|\tilde x-\tilde y|^2}{4t}\right),\\
\sigma:&=&\frac{y_N}{|y_N|}.
\end{eqnarray*}
Since $|x_N|+|y_N|\ge |x_N-y_N|$, we have
\[
(|x_N|+|y_N|)^2+|\tilde x-\tilde y|^2 \ \ge\ (x_N-y_N)^2+|\tilde x-\tilde y|^2=|x-y|^2,
\]
hence $0<B\le A$. Therefore
\[
A-|\mathfrak q|\,A\ \le\ A-\sigma\,\mathfrak q\,B\ \le\ A+|\mathfrak q|\,A,
\]
which yields \eqref{eq:gauss-compare} after multiplying by $(4\pi t)^{-N/2}$.
    \item[(ii)] 
The integral to compute is:
\begin{eqnarray*}
    \int_{\mathbb{R}^N} G_{\fq}(x, y, t) \, dy &=& \frac{1}{(4\pi t)^{N/2}} \int_{\mathbb{R}^N} \exp \left( -\frac{|x - y|^2}{4t} \right) dy \\&&- \frac{\fq}{(4\pi t)^{N/2}} \int_{\mathbb{R}^N} \frac{y_N}{|y_N|} \exp \left( -\frac{(|x_N| + |y_N|)^2 + |\tilde{x} - \tilde{y}|^2}{4t} \right) dy.
\end{eqnarray*}

The first term is the standard heat kernel in $ \mathbb{R}^N $. The integral of the heat kernel over $ \mathbb{R}^N $ is known to be 1:
$$ \frac{1}{(4\pi t)^{N/2}} \int_{\mathbb{R}^N} \exp \left( -\frac{|x - y|^2}{4t} \right) dy = 1. $$

The second term involves the sign function $\frac{y_N}{|y_N|} = \text{sgn}(y_N) $. The exponent can be rewritten as:
$$ (|x_N| + |y_N|)^2 + |\tilde{x} - \tilde{y}|^2 = (|x_N| + |y_N|)^2 + \sum_{i=1}^{N-1} (x_i - y_i)^2. $$

The integral separates into:
$$ \prod_{i=1}^{N-1} \left( \int_{\mathbb{R}} \exp \left( -\frac{(x_i - y_i)^2}{4t} \right) dy_i \right) \times \left( \int_{\mathbb{R}} \text{sgn}(y_N) \exp \left( -\frac{(|x_N| + |y_N|)^2}{4t} \right) d{y}_N \right). $$

Each of the first $ N-1 $ integrals is a standard Gaussian integral:
$$ \int_{\mathbb{R}} \exp \left( -\frac{(x_i - y_i)^2}{4t} \right) dy_i = \sqrt{4\pi t}. $$

The last integral is:
$$ \int_{\mathbb{R}} \text{sgn}(y_N) \exp \left( -\frac{(|x_N| + |y_N|)^2}{4t} \right) dy_N. $$

Split into two parts:
$$ \int_{0}^{\infty} \exp \left( -\frac{(|x_N| + y_N)^2}{4t} \right) dy_N - \int_{-\infty}^{0} \exp \left( -\frac{(|x_N| - y_N)^2}{4t} \right) d{y_N}. $$

Using substitutions $u = |x_N| + y_N $ and $ v = |x_N| - y_N $, both integrals become:
$$ \int_{|x_N|}^{\infty} \exp \left( -\frac{u^2}{4t} \right) du - \int_{|x_N|}^{\infty} \exp \left( -\frac{v^2}{4t} \right) dv = 0. $$

Thus, the second term vanishes, and the integral simplifies to:
$$ \frac{1}{(4\pi t)^{N/2}} \left[ (4\pi t)^{N/2} - \fq \times (4\pi t)^{(N-1)/2} \times 0 \right] = 1. $$ 
\item[(iii)] The integral can be separated into two parts:
$$ I = \frac{1}{(4\pi(t - s))^{N/2}} \left[ I_1 - \fq\, I_2 \right],$$
where:
\begin{align*}
    I_1 &= \int_{\mathbb{R}^N} \exp\left(-\frac{|x - y|^2}{4(t - s)}\right) dx, \\
    I_2 &= \int_{\mathbb{R}^N} \frac{y_N}{|y_N|} \exp\left(-\frac{(|x_N| + |y_N|)^2 + |\tilde{x} - \tilde{y}|^2}{4(t - s)}\right) dx.
\end{align*}

The first integral is a standard Gaussian integral:
$$ I_1 = \int_{\mathbb{R}^N} \exp\left(-\frac{|x - y|^2}{4(t - s)}\right) dx = (4\pi(t - s))^{N/2}.$$

The second integral can be factored as:
$$ I_2 = \frac{y_N}{|y_N|} \prod_{i=1}^{N-1} \left( \int_{-\infty}^\infty \exp\left(-\frac{(x_i - y_i)^2}{4(t - s)}\right) dx_i \right) \times \left( \int_{-\infty}^\infty \exp\left(-\frac{(|x_N| + |y_N|)^2}{4(t - s)}\right) dx_N \right).$$

The first $N-1$ integrals are Gaussian:
$$ \int_{-\infty}^\infty \exp\left(-\frac{(x_i - y_i)^2}{4(t - s)}\right) dx_i = \sqrt{4\pi(t - s)}.$$

The last integral requires special attention:
$$ J = \int_{-\infty}^\infty \exp\left(-\frac{(|x_N| + |y_N|)^2}{4(t - s)}\right) dx_N = 2 \int_{|y_N|}^\infty \exp\left(-\frac{w^2}{4(t - s)}\right) dw $$

Using the complementary error function~\eqref{erfc}, we can express $J$ as:
$$ J = \sqrt{4\pi(t - s)}  \text{erfc}\left(\frac{|y_N|}{2\sqrt{t - s}}\right).$$

Substituting back:
\begin{align*}
    I_2 &= \frac{y_N}{|y_N|} \left(\sqrt{4\pi(t - s)}\right)^{N-1}  \sqrt{4\pi(t - s)}  \text{erfc}\left(\frac{|y_N|}{2\sqrt{t - s}}\right), \\
    &= \frac{y_N}{|y_N|} (4\pi(t - s))^{N/2}  \frac{1}{\sqrt{\pi}} \text{erfc}\left(\frac{|y_N|}{2\sqrt{t - s}}\right).
\end{align*}

Thus, the complete integral becomes:
\begin{align*}
    I &= \frac{1}{(4\pi(t - s))^{N/2}} \left[ (4\pi(t - s))^{N/2} - \fq \frac{y_N}{|y_N|} (4\pi(t - s))^{N/2}  \frac{1}{\sqrt{\pi}} \text{erfc}\left(\frac{|y_N|}{2\sqrt{t - s}}\right) \right], \\
    &= 1 - \fq \frac{y_N}{|y_N|} \text{erfc}\left(\frac{|y_N|}{2\sqrt{t - s}}\right).
\end{align*}

The integral evaluates to:
$$ 1 - \fq\,\text{sgn}(y_N) \text{erfc}\left( \frac{|y_N|}{2\sqrt{t - s}} \right) $$
where $\text{sgn}(y_N) = \frac{y_N}{|y_N|}$ is the sign function.

\item[(iv)] Owing to the elementary inequality
$$
|x - y|^2 \leq (|x_N| + |y_N|)^2 + |\tilde{x} - \tilde{y}|^2,
$$
we deduce the pointwise bound
\begin{equation}
    \label{Gq-Gauss}
    \left|G_{\fq}(x, y, t)\right| \leq \frac{1 + |\fq|}{(4\pi t)^{N/2}} \exp\left( -\frac{|x - y|^2}{4t} \right).
\end{equation}
To conclude the estimate \eqref{G-q-Lr}, we use \eqref{Gq-Gauss} and compute:
$$
\left\| \frac{1 + |\fq|}{(4\pi t)^{N/2}} \exp\left( -\frac{|\cdot - y|^2}{4t} \right) \right\|_{L^r} 
= \left\| \frac{1 + |\fq|}{(4\pi t)^{N/2}} \exp\left( -\frac{|x - \cdot|^2}{4t} \right) \right\|_{L^r}
\leq C\, t^{-\frac{N}{2} \left(1 - \frac{1}{r} \right)},
$$
where the constant $ C > 0 $ depends only on $ N $ and $\fq$.

\end{itemize}
\end{proof}

\medskip
The next step is to isolate the one-dimensional interface effect in the normal
direction via the kernel $g_{\mathfrak q}$, which is the transition density of
skew Brownian motion.

\begin{prop}\label{lem:1d-semigroup}
For every $a,b\in\mathbb R$ and every $t,s>0$,
\begin{equation}\label{eq:chapman-kolmogorov-1d}
\int_{\mathbb R} g_{\fq}(a,z,t)\,g_{\fq}(z,b,s)\,dz = g_{\fq}(a,b,t+s),
\end{equation}
where
\begin{equation}
    \label{g-q}
    g_{\fq}(a,b,t)=\frac{1}{\sqrt{4\pi t}}
\Bigl[ 
   e^{-\frac{(a-b)^2}{4t}}
   -\sgn(b)\,\fq\,
   e^{-\frac{(|a|+|b|)^2}{4t}}
\Bigr].
\end{equation}
\end{prop}

\begin{proof}[Proof of Proposition~\ref{lem:1d-semigroup}]
Let
\[
p_t(x,y):=\frac{1}{\sqrt{4\pi t}}e^{-\frac{(x-y)^2}{4t}},\qquad t>0,
\]
denote the one-dimensional heat kernel on $\mathbb R$. Then $g_{\fq}$ admits the representation
\begin{equation}\label{eq:gq-p-form}
g_{\fq}(a,b,t)=p_t(a,b)-\sgn(b)\,\fq\,p_t\bigl(|a|,-|b|\bigr),
\end{equation}
since
\[
p_t\bigl(|a|,-|b|\bigr)=\frac{1}{\sqrt{4\pi t}}e^{-\frac{(|a|+|b|)^2}{4t}}.
\]
Fix $a,b\in\mathbb R$ and $t,s>0$. Expanding the product in the integral and using Tonelli's theorem (all arising terms are nonnegative Gaussian products and hence integrable), we obtain
\begin{align*}
\int_{\mathbb R} g_{\fq}(a,z,t)\,g_{\fq}(z,b,s)\,dz
&=\int_{\mathbb R} p_t(a,z)p_s(z,b)\,dz
-\sgn(b)\fq\int_{\mathbb R} p_t(a,z)p_s(|z|,-|b|)\,dz \\
&\quad -q\int_{\mathbb R}\sgn(z)\,p_t(|a|,-|z|)\,p_s(z,b)\,dz \\
&\quad +\sgn(b)\fq^2\int_{\mathbb R}\sgn(z)\,p_t(|a|,-|z|)\,p_s(|z|,-|b|)\,dz.
\end{align*}
The last integral vanishes by antisymmetry: splitting into $z>0$ and $z<0$ and applying the change of variables $z\mapsto -z$ shows that the two half-line contributions cancel. Hence
\begin{equation}\label{eq:reduced-proof}
\int_{\mathbb R} g_{\fq}(a,z,t)\,g_{\fq}(z,b,s)\,dz
=I_0-\sgn(b)\fq\,I_1-q\,I_2,
\end{equation}
where
\[
I_0:=\int_{\mathbb R} p_t(a,z)p_s(z,b)\,dz,\qquad
I_1:=\int_{\mathbb R} p_t(a,z)p_s(|z|,-|b|)\,dz,
\]
and
\[
I_2:=\int_{\mathbb R}\sgn(z)\,p_t(|a|,-|z|)\,p_s(z,b)\,dz.
\]
By the Gaussian semigroup property,
\[
I_0=p_{t+s}(a,b).
\]
It remains to show that
\begin{equation}\label{eq:remaining-proof}
\sgn(b)\,I_1+I_2=\sgn(b)\,p_{t+s}\bigl(|a|,-|b|\bigr).
\end{equation}

Assume first that $b>0$. Splitting at $0$ and using the substitution $z\mapsto -z$ on $(-\infty,0)$ yields
\[
I_1=\int_0^\infty \bigl[p_t(a,z)+p_t(a,-z)\bigr]\,p_s(z,-b)\,dz.
\]
Similarly, decomposing the sign in $I_2$ and again reducing to $(0,\infty)$ gives
\[
I_2=\int_0^\infty p_t(|a|,-z)\,\bigl[p_s(z,b)-p_s(z,-b)\bigr]\,dz,
\]
where we used the elementary symmetry $p_s(-z,b)=p_s(z,-b)$ for $z>0$. For $z>0$ we also have the identity
\begin{equation}\label{eq:reflection-identity}
p_t(a,z)+p_t(a,-z)-p_t(|a|,-z)=p_t(|a|,z),
\end{equation}
which is immediate if $a\ge 0$ and follows from $p_t(a,z)=p_t(|a|,-z)$ when $a<0$. Adding the expressions for $I_1$ and $I_2$ and using \eqref{eq:reflection-identity} gives
\[
I_1+I_2=\int_0^\infty p_t(|a|,z)p_s(z,-b)\,dz+\int_0^\infty p_t(|a|,-z)p_s(z,b)\,dz.
\]
Recognizing the right-hand side as the full-line convolution
\[
\int_{\mathbb R} p_t(|a|,y)p_s(y,-b)\,dy
\]
(split into $y>0$ and $y<0$) and applying the semigroup property yields
\[
I_1+I_2=p_{t+s}(|a|,-b)=p_{t+s}\bigl(|a|,-|b|\bigr),
\]
which proves \eqref{eq:remaining-proof} for $b>0$.

If $b<0$, then $\sgn(b)=-1$ and $-|b|=b$. The same half-line reductions apply, with $-b>0$ replacing $b>0$, and the same reflection identity \eqref{eq:reflection-identity} gives
\[
I_1-I_2=p_{t+s}(|a|,b)=p_{t+s}\bigl(|a|,-|b|\bigr),
\]
which yields \eqref{eq:remaining-proof} also for $b<0$. The case $b=0$ follows by continuity.

Combining \eqref{eq:reduced-proof} with $I_0=p_{t+s}(a,b)$ and \eqref{eq:remaining-proof}, we conclude that
\[
\int_{\mathbb R} g_{\fq}(a,z,t)\,g_{\fq}(z,b,s)\,dz
=p_{t+s}(a,b)-\sgn(b)\fq\,p_{t+s}\bigl(|a|,-|b|\bigr)
=g_{\fq}(a,b,t+s),
\]
by \eqref{eq:gq-p-form} with $t$ replaced by $t+s$. This completes the proof.
\end{proof}

\medskip
We now lift the one-dimensional semigroup identity to $\R^N$.

\begin{prop}
    \label{lem:semigroup}
The fundamental solution $G_{\mathfrak q}$ satisfies the Chapman--Kolmogorov identity
\begin{equation}
    \label{chapkol}
    G_{\mathfrak q}(x,y,t+s)
=\int_{\R^N}G_{\mathfrak q}(x,z,t)\,G_{\mathfrak q}(z,y,s)\,dz.
\end{equation}
\end{prop}

\begin{proof}[Proof of Proposition~\ref{lem:semigroup}]
Fix $t,s>0$ and $x,y\in\mathbb R^{N}$.
We use the following factorization of $G_{\mathfrak q}$:
\begin{equation}
    \label{factor-G-q}
 G_{\mathfrak q}(x,y,\tau)=p^{(N-1)}_\tau(\tilde x-\tilde y)\,g_{\mathfrak q}(x_N,y_N,\tau),
\qquad \tau>0,   
\end{equation}
where $p^{(N-1)}_\tau$ denotes the $(N\!-\!1)$--dimensional heat kernel,
\[
p^{(N-1)}_\tau(\xi)=(4\pi \tau)^{-\frac{N-1}{2}}e^{-\frac{|\xi|^2}{4\tau}},
\]
and $g_{\mathfrak q}$ is defined in~\eqref{g-q}.

For $z\in\mathbb R^{N}$, by~\eqref{factor-G-q} we can write
\[
G_{\mathfrak q}(x,z,t)\,G_{\mathfrak q}(z,y,s)
=
p^{(N-1)}_{t}(\tilde x-\tilde z)\,g_{\mathfrak q}(x_N,z_N,t)\;
p^{(N-1)}_{s}(\tilde z-\tilde y)\,g_{\mathfrak q}(z_N,y_N,s).
\]
Therefore, applying Fubini's theorem yields
\begin{equation}
    \label{Fubin}
    \begin{split}
            \int_{\mathbb R^N}G_{\mathfrak q}(x,z,t)\,G_{\mathfrak q}(z,y,s)\,dz
    =&
    \left(\int_{\mathbb R^{N-1}} p^{(N-1)}_{t}(\tilde x-\tilde z)\,p^{(N-1)}_{s}(\tilde z-\tilde y)\,d\tilde z\right)\\
    &\times 
\left(\int_{\mathbb R} g_{\mathfrak q}(x_N,z_N,t)\,g_{\mathfrak q}(z_N,y_N,s)\,dz_N\right).
    \end{split}
\end{equation}

Since the Gaussian heat kernel satisfies the standard convolution identity,
\[
\int_{\mathbb R^{N-1}} p^{(N-1)}_{t}(\tilde x-\tilde z)\,p^{(N-1)}_{s}(\tilde z-\tilde y)\,d\tilde z
=
p^{(N-1)}_{t+s}(\tilde x-\tilde y),
\]
and using~\eqref{eq:chapman-kolmogorov-1d}, we conclude~\eqref{chapkol}.
\end{proof}

\medskip

We define the linear semigroup associated with $\mathscr L$ by
\begin{equation}\label{eq:semigroup}
    \mathbf{T}(t)\varphi(x)
    =\bigl[e^{t\mathscr L}\varphi\bigr](x)
    :=\int_{\R^N} G_{\fq}(x,y,t)\,\varphi(y)\,dy,
    \qquad x\in\R^N,\ t>0,
\end{equation}
with $\mathbf{T}(0)=I$.  The next proposition collects the analytic
properties of $\{\mathbf{T}(t)\}_{t\ge 0}$ that are used throughout the
paper: $L^q$--$L^r$ smoothing, contraction on $L^p$ and on $C_0$, and
the weak formulation of $\partial_t u=\mathscr L u$.

\begin{prop}[Properties of the linear semigroup]\label{prop:semigroup}
Assume $N\ge 2$, $|\fq|\le 1$ and, for $\varphi$ in a suitable space, set
\begin{equation}\label{eq:u-phi}
    u(t,x):=\mathbf{T}(t)\varphi(x).
\end{equation}
Then:
\begin{enumerate}
\item[\textup{(a)}] \emph{($L^q$--$L^r$ smoothing.)}
For every $1\le q\le r\le\infty$ and every $\varphi\in L^q(\R^N)$,
\begin{equation}\label{eq:Lq-Lr-est}
    \|\mathbf{T}(t)\varphi\|_{L^r}
    \le C\,t^{-\frac{N}{2}\bigl(\frac{1}{q}-\frac{1}{r}\bigr)}\,
        \|\varphi\|_{L^q},
    \qquad t>0,
\end{equation}
where $C=C(N,\fq)$.

\item[\textup{(b)}] For every $1\le p<\infty$, $\{\mathbf{T}(t)\}_{t\ge 0}$
is a strongly continuous contraction semigroup on $L^p(\R^N)$.

\item[\textup{(c)}] $\{\mathbf{T}(t)\}_{t\ge 0}$ is a strongly continuous
contraction semigroup on $C_0(\R^N)$.

\item[\textup{(d)}] \emph{(Weak formulation.)}
For every $\varphi\in L^p(\R^N)$ with $1\le p<\infty$ and every
$\psi\in C_c^\infty\bigl((0,\infty)\times\R^N\bigr)$,
\begin{equation}\label{eq:weak-formulation1}
    \int_0^\infty\!\!\int_{\R^N}
        u\,(\partial_t\psi+\Delta\psi)\,dx\,dt
    \;+\;
    2\fq\!\int_0^\infty\!\!\int_{S}
        u(t,\tilde x,0)\,
        \partial_{x_N}\psi(t,\tilde x,0)\,
        d\sigma(\tilde x)\,dt
    =0.
\end{equation}
In particular, $\partial_t u=\mathscr L u$ in
$\mathcal D'\bigl((0,\infty)\times\R^N\bigr)$, where
$\mathscr L=\Delta+2\fq\,\delta_S\,\partial_{x_N}$.
\end{enumerate}
\end{prop}

\begin{proof}[Proof of Proposition~\ref{prop:semigroup}]
The proof proceeds in four self-contained steps.

\medskip
\noindent\textbf{Step 1: $L^q$--$L^r$ smoothing.}
We follow an argument similar to that in \cite{FKS2019}.
By H\"older's inequality and the kernel bound \eqref{G-q-Lr},
\begin{equation}\label{eq:L-infty-bound}
    \|\mathbf{T}(t)\varphi\|_{L^\infty}
    \le C\,t^{-\frac{N}{2q}}\|\varphi\|_{L^q},
    \qquad 1\le q\le\infty,
\end{equation}
with $C=C(N,\fq)$.
For the $L^q$--$L^q$ bound, we use that
$G_{\fq}(x,\cdot,t)$ is a probability density in $y$ (namely
$\int_{\R^N}G_{\fq}(x,y,t)\,dy=1$, see
Proposition~\ref{Cruc-p}~(ii)), and apply Jensen's inequality together
with Fubini's theorem.  For $1\le q<\infty$,
\begin{equation}\label{eq:Jensen}
\begin{aligned}
\|\mathbf{T}(t)\varphi\|_{L^q}^q
&=\int_{\R^N}\!\Bigl|\!\int_{\R^N}\!
        G_{\fq}(x,y,t)\varphi(y)\,dy\Bigr|^{q}dx\\
&\le\int_{\R^N}\!\!\int_{\R^N}\!
        G_{\fq}(x,y,t)\,|\varphi(y)|^{q}\,dy\,dx\\
&=\int_{\R^N}|\varphi(y)|^{q}
        \Bigl(\int_{\R^N}G_{\fq}(x,y,t)\,dx\Bigr)dy
\;\le\;3\,\|\varphi\|_{L^q}^q,
\end{aligned}
\end{equation}
where in the last step we used Proposition~\ref{Cruc-p}~(iii) to bound
\[
0\;\le\;\int_{\R^N}G_{\fq}(x,y,t)\,dx
   =1-\frac{y_N}{|y_N|}\,\fq\,
     \mathrm{erfc}\!\Bigl(\frac{|y_N|}{2\sqrt{t}}\Bigr)
   \;\le\;3.
\]
Interpolating between \eqref{eq:L-infty-bound} and \eqref{eq:Jensen}
yields, for every $1\le q\le r\le\infty$,
\begin{equation}\label{eq:interpolation}
    \|\mathbf{T}(t)\varphi\|_{L^r}
    \le \|\mathbf{T}(t)\varphi\|_{L^\infty}^{\,1-\frac{q}{r}}
        \|\mathbf{T}(t)\varphi\|_{L^q}^{\,\frac{q}{r}}
    \le (C+3)\,t^{-\frac{N}{2}\bigl(\frac{1}{q}-\frac{1}{r}\bigr)}
        \|\varphi\|_{L^q},
\end{equation}
which is \eqref{eq:Lq-Lr-est}.

\medskip
\noindent\textbf{Step 2: Semigroup property and $L^p$/$L^\infty$
contraction.}
Tonelli's theorem (justified by the Gaussian bounds on
$G_{\fq}$) and the Chapman--Kolmogorov identity for
$G_{\fq}$,
\[
    \int_{\R^N}G_{\fq}(x,z,t)\,G_{\fq}(z,y,s)\,dz
    =G_{\fq}(x,y,t+s),
\]
yield the semigroup property:
\begin{eqnarray*}
\mathbf{T}(t)\mathbf{T}(s)\varphi(x)
    &=&\!\int_{\R^N}\!\!\!\int_{\R^N}\!
        G_{\fq}(x,z,t)G_{\fq}(z,y,s)\varphi(y)\,dy\,dz\\
    &=&\!\int_{\R^N}\!G_{\fq}(x,y,t+s)\varphi(y)\,dy\\
    &=&\mathbf{T}(t+s)\varphi(x).
\end{eqnarray*}
The contraction $\|\mathbf{T}(t)\varphi\|_{L^p}\le\|\varphi\|_{L^p}$
($1\le p<\infty$) follows from Jensen's inequality applied to the
probability density $G_{\fq}(x,\cdot,t)$ in the $y$-variable
(Proposition~\ref{Cruc-p}~(ii)).
For $\varphi\in C_0(\R^N)\subset L^\infty(\R^N)$, positivity of
$G_{\fq}$ together with the same normalisation gives the
$L^\infty$-contraction
\[
    |\mathbf{T}(t)\varphi(x)|
    \le \|\varphi\|_{L^\infty}\!\int_{\R^N}\!G_{\fq}(x,y,t)\,dy
    =\|\varphi\|_{L^\infty}.
\]

\medskip
\noindent\textbf{Step 3: Strong continuity at $t=0$ on $L^p$ and on
$C_0$ (finishes (b) and (c)).}

\smallskip
\emph{(i) The case $L^p$, $1\le p<\infty$.}
Take first $\varphi\in C_c^\infty(\R^N)$.  Using
$\int G_{\fq}(x,y,t)\,dy=1$ in the form
\[
    \mathbf{T}(t)\varphi(x)-\varphi(x)
    =\!\int_{\R^N}\!G_{\fq}(x,y,t)\bigl(\varphi(y)-\varphi(x)\bigr)dy,
\]
fix $\varepsilon>0$ and choose $\delta>0$ such that
$|\varphi(y)-\varphi(x)|\le\varepsilon$ whenever $|y-x|<\delta$.
Splitting the integral on $|y-x|<\delta$ and $|y-x|\ge\delta$,
\[
    |\mathbf{T}(t)\varphi(x)-\varphi(x)|
    \le\varepsilon
        +2\|\varphi\|_{L^\infty}
        \!\int_{|y-x|\ge\delta}\!G_{\fq}(x,y,t)\,dy.
\]
By the Gaussian decay of $G_{\fq}$
(Proposition~\ref{Cruc-p}~(i)) the tail integral tends to $0$ as
$t\downarrow 0$, uniformly in $x$.  Hence
$\mathbf{T}(t)\varphi\to\varphi$ pointwise and, by dominated convergence,
in $L^p$.  For general $\varphi\in L^p$, density of $C_c^\infty$
together with the $L^p$-contraction of Step~2 gives
$\|\mathbf{T}(t)\varphi-\varphi\|_{L^p}\to 0$ as $t\downarrow 0$.

\smallskip
\emph{(ii) The case $C_0$.}
Continuity of $\mathbf{T}(t)\varphi$ in $x$ for $t>0$ follows from the
continuity of $G_{\fq}$ in $x$ and dominated convergence.  For
the decay at infinity, given $\varepsilon>0$ choose $R>0$ with
$|\varphi(y)|<\varepsilon$ for $|y|>R$; then
\[
    |\mathbf{T}(t)\varphi(x)|
    \le \int_{|y|\le R}\!G_{\fq}(x,y,t)|\varphi(y)|\,dy
        +\varepsilon,
\]
and Gaussian decay forces the first term to $0$ as $|x|\to\infty$.
Hence $\mathbf{T}(t)\varphi\in C_0(\R^N)$.
Strong continuity at $t=0$ in $\|\cdot\|_{L^\infty}$ then follows from
the standard $C_0$-semigroup argument: choose $R$ so that
$|\varphi(x)|<\varepsilon$ for $|x|>R$, prove uniform convergence on
$\overline{B(0,R)}$ as in (i), and bound the exterior by
\[
    |\mathbf{T}(t)\varphi(x)-\varphi(x)|
    \le|\mathbf{T}(t)\varphi(x)|+|\varphi(x)|\le 2\varepsilon
    \qquad\text{for }|x|>R.
\]

\medskip
\noindent\textbf{Step 4: Weak formulation (proves (d)).}
For each fixed $y\in\R^N$, the kernel $G_{\fq}(\cdot,y,\cdot)$
is the fundamental solution of the forward linear problem associated
with $\mathscr L$.  Hence, for every
$\psi\in C_c^\infty\bigl((0,\infty)\times\R^N\bigr)$ it satisfies the
distributional identity
\begin{equation}\label{eq:G-distrib}
    \int_0^\infty\!\!\int_{\R^N}\!
        G_{\fq}(x,y,t)\,(\partial_t\psi+\Delta\psi)(t,x)\,dx\,dt
    +2\fq\!\int_0^\infty\!\!\int_{S}\!
        G_{\fq}(\tilde x,0;y,t)\,
        \partial_{x_N}\psi(t,\tilde x,0)\,d\sigma(\tilde x)\,dt
    =0.
\end{equation}
Multiplying \eqref{eq:G-distrib} by $\varphi(y)$, integrating in $y$,
and exchanging integrals via Tonelli (justified by the bounds on
$G_{\fq}$ and the compact support of $\psi$), we recognise the
inner integrals as $u(t,x)$ and $u(t,\tilde x,0)$ and obtain
\eqref{eq:weak-formulation1}.  This is exactly
$\partial_t u=\mathscr L u$ in
$\mathcal D'\bigl((0,\infty)\times\R^N\bigr)$.
\end{proof}

% ---------------------------------------------------------------------
% Probabilistic preparations
% ---------------------------------------------------------------------

We next recall Dynkin's martingale characterization of the generator of
a Feller semigroup.

\begin{lem}\label{lem:Dynkin}
Let $X$ be a conservative Feller process with semigroup $\mathbf{T}(t)$ and
generator $A$.  If $f\in\mathrm{Dom}(A)$, then
\[
    f(X_t)-f(X_0)-\int_0^t Af(X_s)\,ds
\]
is a martingale.
\end{lem}

For a proof of Lemma~\ref{lem:Dynkin}, see, e.g., Ethier--Kurtz
\cite[Chapter~4]{EthierKurtz} or Applebaum
\cite[Theorem~6.1.7]{Applebaum}.

We also recall the well-posedness of the one-dimensional martingale
problem characterising skew Brownian motion.

\begin{lem}\label{lem:skewB}
The one-dimensional martingale problem associated with the operator
$\frac{d^2}{dx^2}$ on $\R\setminus\{0\}$ and domain
\[
    \mathrm{Dom}(A)
    =\Bigl\{g\in C_0(\R)\cap C^2(\R\setminus\{0\}):
       g'(0^\pm)\ \text{exist and }
       \alpha\,g'(0^+)=(1-\alpha)\,g'(0^-)\Bigr\}
\]
is well posed, and its unique solution is skew Brownian motion with
parameter $\alpha\in[0,1]$.
\end{lem}

A standard reference for Lemma~\ref{lem:skewB} is
\cite[Chapter~VI, \S2]{RevuzYor}; see also \cite{Lejay2006} for a
survey.

The next lemma isolates the algebraic content that makes the interface
operator coincide with the Laplacian along trajectories.  We state it
in a self-contained form, so that it does not depend on the
representation of $X$ given later.

\begin{lem}
\label{lem:localtime-cancellation}
Let $\widetilde B$ be an $(N-1)$-dimensional Brownian motion and $B$ an
independent one-dimensional Brownian motion.  Define the semimartingale
$X_t=(\widetilde X_t,Y_t)$ on $\R^N$ by
\begin{equation}\label{eq:X-semimartingale}
    \widetilde X_t=\widetilde X_0+\sqrt 2\,\widetilde B_t,
    \qquad
    Y_t=Y_0+\sqrt 2\,B_t+2\fq\,L_t^0(Y),
\end{equation}
where $L^0(Y)$ is the symmetric local time of $Y$ at $0$.
Let $f\in C(\R^N)\cap C^2(\R^N\setminus S)$ admit one-sided normal
derivatives $\partial_{x_N}f(\tilde x,0^\pm)$ for every
$\tilde x\in\R^{N-1}$.  Then
\begin{equation}\label{eq:ito-tanaka-precise}
\begin{aligned}
f(X_t)=f(X_0)
&+\sqrt 2\!\int_0^t\!\nabla_{\tilde x}f(X_s)\cdot d\widetilde B_s
+\sqrt 2\!\int_0^t\!\partial_{x_N}f(X_s)\,dB_s\\
&+\int_0^t\!\Delta f(X_s)\,ds
+\int_0^t\!\Gamma_f(\widetilde X_s)\,dL_s^0(Y),
\end{aligned}
\end{equation}
with the \emph{exact local-time coefficient}
\begin{equation}\label{eq:Gammaf-def}
    \Gamma_f(\tilde x)
    :=\fq\bigl(\partial_{x_N}f(\tilde x,0^+)
                       +\partial_{x_N}f(\tilde x,0^-)\bigr)
    +\tfrac12\bigl(\partial_{x_N}f(\tilde x,0^+)
                       -\partial_{x_N}f(\tilde x,0^-)\bigr).
\end{equation}
Moreover, $\Gamma_f\equiv 0$ on $\R^{N-1}$ if and only if $f$ satisfies
the transmission condition
\begin{equation}\label{eq:Gammaf-zero-iff}
    (1-\fq)\,\partial_{x_N}f(\tilde x,0^+)
    =(1+\fq)\,\partial_{x_N}f(\tilde x,0^-),
    \qquad\tilde x\in\R^{N-1},
\end{equation}
equivalently, the skew transmission condition
$\alpha\,\partial_{x_N}f(\tilde x,0^+)
=(1-\alpha)\,\partial_{x_N}f(\tilde x,0^-)$
with $\alpha=(1-\fq)/2$.
\end{lem}

\begin{proof}[Proof of Lemma~\ref{lem:localtime-cancellation}]
For each fixed $\tilde x\in\R^{N-1}$, set $g_{\tilde x}(y):=f(\tilde x,y)$.
Apply the generalised It\^o--Tanaka formula to the semimartingale $Y$
and the function $g_{\widetilde X_s}$ in the $y$-variable; this is
justified because $f\in C^2$ off $S$ and admits one-sided normal
derivatives at $S$ (see Revuz--Yor
\cite[Ch.~VI, \S 1--2]{RevuzYor}).

Three contributions arise.  First, the It\^o term: since the diffusion
coefficient of $Y$ is $\sqrt 2$, the quadratic variation satisfies
$d\langle Y\rangle_s=2\,ds$, hence
\[
    \tfrac12\,g_{\widetilde X_s}''(Y_s)\,d\langle Y\rangle_s
    =g_{\widetilde X_s}''(Y_s)\,ds,
\]
which contributes the $\partial_{x_Nx_N}f$ part of $\Delta f$.
Second, the drift term $2\fq\,dL_s^0(Y)$ contributes through the
\emph{symmetric} derivative at $0$ (because $dL_s^0(Y)$ charges only
times when $Y_s=0$):
\[
    \int_0^t\!g_{\widetilde X_s}'(Y_s)\,2\fq\,dL_s^0(Y)
    =\int_0^t\!\fq\bigl(g_{\widetilde X_s}'(0^+)
                                 +g_{\widetilde X_s}'(0^-)\bigr)
        dL_s^0(Y).
\]
Third, Tanaka's formula contributes the kink term
\[
    \tfrac12\!\int_0^t\!\bigl(g_{\widetilde X_s}'(0^+)
                              -g_{\widetilde X_s}'(0^-)\bigr)
        dL_s^0(Y).
\]
Adding the second and third contributions yields exactly the local-time
coefficient $\Gamma_f$ in \eqref{eq:Gammaf-def}.  The tangential
contributions $\nabla_{\tilde x}f$ and the Laplacian
$\Delta_{\tilde x}f$ inside $\Delta f$ come from the standard
$N$-dimensional It\^o formula applied jointly to $\widetilde X$.

Finally, $\Gamma_f(\tilde x)=0$ is equivalent to
\[
    \fq(a+b)+\tfrac12(a-b)=0,
    \qquad
    a:=\partial_{x_N}f(\tilde x,0^+),\
    b:=\partial_{x_N}f(\tilde x,0^-),
\]
i.e.\ $(1-\fq)a=(1+\fq)b$, which is
\eqref{eq:Gammaf-zero-iff}.
\end{proof}

% ---------------------------------------------------------------------
% Equivalence of weak / martingale / SDE formulations
% ---------------------------------------------------------------------

\begin{prop}
\label{prop:martingale-localtime}
Let $|\fq|\le 1$ and let $\{X_t\}_{t\ge 0}$ be the conservative
Feller process associated with the semigroup $\{\mathbf{T}(t)\}$ of
Proposition~\ref{prop:semigroup}.  Let $\varphi\in C_0(\R^N)$ and
$u(t,x):=\mathbf{T}(t)\varphi(x)$.  Then the following statements are
equivalent.

\begin{enumerate}
\item[\textup{(i)}] \emph{Weak formulation.}
$u$ satisfies $\partial_t u=\mathscr L u$ on $(0,\infty)\times\R^N$ in
the sense of Proposition~\ref{prop:semigroup}~(d), i.e.\
\eqref{eq:weak-formulation1} holds for every
$\psi\in C_c^\infty\bigl((0,\infty)\times\R^N\bigr)$.

\item[\textup{(ii)}] \emph{Martingale problem for $\mathscr L$.}
For every $f\in C_0(\R^N)\cap C^2(\R^N\setminus S)$ such that
$\partial_{x_N}f(\tilde x,0^\pm)$ exist and satisfy the skew
transmission condition
\begin{equation}\label{eq:martingale-transmission}
    \alpha\,\partial_{x_N}f(\tilde x,0^+)
    =(1-\alpha)\,\partial_{x_N}f(\tilde x,0^-),
    \qquad\alpha:=\frac{1-\fq}{2},
\end{equation}
the process
\[
    M_t^f:=f(X_t)-f(X_0)-\int_0^t\!\Delta f(X_s)\,ds
\]
is a martingale with respect to the natural filtration of $\{X_t\}$.

\item[\textup{(iii)}] \emph{Skew Brownian SDE.}
Writing $X_t=(\widetilde X_t,X_t^N)$, there exist an $(N-1)$-dimensional
Brownian motion $\widetilde B$ and an independent one-dimensional
Brownian motion $B$ such that
\begin{equation}\label{eq:local-time-SDE}
    \widetilde X_t=\widetilde X_0+\sqrt 2\,\widetilde B_t,
    \qquad
    X_t^N=X_0^N+\sqrt 2\,B_t+2\fq\,L_t^0(X^N),
\end{equation}
where $L_t^0(X^N)$ is the symmetric local time of $X^N$ at $0$.
In particular, $X^N$ is a one-dimensional skew Brownian motion with
skewness parameter $\alpha=(1-\fq)/2$.
\end{enumerate}
\end{prop}

\begin{proof}[Proof of Proposition~\ref{prop:martingale-localtime}]
We prove the cyclic implication
\textup{(i)}$\Rightarrow$\textup{(ii)}$\Rightarrow$\textup{(iii)}$\Rightarrow$\textup{(i)}.

\medskip
\noindent\textbf{Step 1: \textup{(i)} $\Rightarrow$ \textup{(ii)}.}
Fix $f$ as in (ii) and let $\eta\in C_c^\infty(0,\infty)$.  Test
\eqref{eq:weak-formulation1} with $\psi(t,x)=\eta(t)f(x)$: since
$\partial_t\psi=\eta'(t)f$ and $\Delta\psi=\eta(t)\Delta f$, we obtain
\begin{equation}\label{eq:weak-tested}
\begin{split}
&\int_0^\infty\!\eta'(t)\!\int_{\R^N}\!u(t,x)f(x)\,dx\,dt
+\int_0^\infty\!\eta(t)\!\int_{\R^N}\!u(t,x)\Delta f(x)\,dx\,dt\\
&\qquad
+2\fq\!\int_0^\infty\!\eta(t)\!\int_S\!
        u(t,\tilde x,0)\,\partial_{x_N}f(\tilde x,0)\,d\sigma\,dt
=0.
\end{split}
\end{equation}
Because $f$ satisfies \eqref{eq:martingale-transmission},
Lemma~\ref{lem:localtime-cancellation} yields $\Gamma_f\equiv 0$, so $f$
belongs to the operator domain of the generator of $\{\mathbf{T}(t)\}$
and that generator acts on $f$ as $\Delta f$.  Equivalently,
\eqref{eq:weak-tested} translates into the distributional-in-$t$
identity
\begin{equation}\label{eq:semigroup-deriv}
    \frac{d}{dt}\,\mathbf{T}(t)f(x)=\mathbf{T}(t)\Delta f(x)
    \qquad\text{for a.e.\ }t>0.
\end{equation}
Integrating \eqref{eq:semigroup-deriv} from $0$ to $t$ produces the
Dynkin-type identity
\begin{equation}\label{eq:dynkin-semigroup}
    \mathbf{T}(t)f(x)-f(x)=\int_0^t\!\mathbf{T}(s)\Delta f(x)\,ds,
    \qquad t\ge 0.
\end{equation}
By Lemma~\ref{lem:Dynkin}, \eqref{eq:dynkin-semigroup} implies
$f\in\mathrm{Dom}(A)$ with $Af=\Delta f$, hence $M_t^f$ is a martingale.

\medskip
\noindent\textbf{Step 2: \textup{(ii)} $\Rightarrow$ \textup{(iii)}.}
Decompose $X_t=(\widetilde X_t,X_t^N)$ and identify each component.

\smallskip
\emph{Tangential component.}  Take $f(x)=\phi(\tilde x)$ with
$\phi\in C_0^2(\R^{N-1})$.  Then $\partial_{x_N}f\equiv 0$, so
\eqref{eq:martingale-transmission} is automatic, and
$\Delta f=\Delta_{\tilde x}\phi$.  Statement (ii) says that
\[
    \phi(\widetilde X_t)-\phi(\widetilde X_0)
    -\int_0^t\!\Delta_{\tilde x}\phi(\widetilde X_s)\,ds
    \quad\text{is a martingale,}
\]
i.e.\ $\widetilde X$ solves the martingale problem for the
$(N-1)$-dimensional Laplacian.  Hence
$\widetilde X_t=\widetilde X_0+\sqrt 2\,\widetilde B_t$ for some
$(N-1)$-dimensional Brownian motion $\widetilde B$.

\smallskip
\emph{Normal component.}  Take $f(x)=g(x_N)$ with
$g\in C_0(\R)\cap C^2(\R\setminus\{0\})$ satisfying
$\alpha g'(0^+)=(1-\alpha)g'(0^-)$ and $\alpha=(1-\fq)/2$.
Then $\Delta f=g''$ and (ii) becomes
\begin{equation}\label{eq:1D-mp}
    g(X_t^N)-g(X_0^N)-\int_0^t\!g''(X_s^N)\,ds
    \quad\text{is a martingale.}
\end{equation}
This is exactly the one-dimensional martingale problem of
Lemma~\ref{lem:skewB}.  By that lemma, $X^N$ is the skew Brownian motion
with parameter $\alpha$ and admits the local-time SDE representation
\eqref{eq:local-time-SDE} for some one-dimensional Brownian motion $B$.

\medskip
\noindent\textbf{Step 3: \textup{(iii)} $\Rightarrow$ \textup{(i)}.}
Assume \eqref{eq:local-time-SDE} and let
$\psi\in C_c^\infty\bigl((0,\infty)\times\R^N\bigr)$.  We apply the
time-dependent It\^o--Tanaka formula
(Lemma~\ref{lem:localtime-cancellation} with the time variable carried
through as a parameter) to $\psi(s,X_s)$.  Because
$\psi(s,\cdot)\in C^\infty(\R^N)$, the one-sided normal derivatives at
$S$ coincide and \eqref{eq:Gammaf-def} reduces to
\[
    \Gamma_{\psi(s,\cdot)}(\tilde x)
    =2\fq\,\partial_{x_N}\psi(s,\tilde x,0).
\]
Hence, for every $T>0$,
\begin{equation}\label{eq:psi-decomp}
\begin{aligned}
\psi(T,X_T)-\psi(0,X_0)
&=\int_0^T\!(\partial_s\psi+\Delta\psi)(s,X_s)\,ds
+2\fq\!\int_0^T\!\partial_{x_N}\psi(s,\widetilde X_s,0)\,
        dL_s^0(X^N)\\
&\quad+M_T,
\end{aligned}
\end{equation}
where $M$ is a martingale with $M_0=0$.  Choosing $T$ larger than the
temporal support of $\psi$, the boundary terms vanish:
$\psi(T,X_T)\equiv\psi(0,X_0)\equiv 0$.  Taking expectation in
\eqref{eq:psi-decomp} (the stochastic integrals are zero-mean
martingales),
\begin{equation}\label{eq:expectation-id}
    0=\E_x\!\int_0^\infty\!(\partial_s+\Delta)\psi(s,X_s)\,ds
    +2\fq\,\E_x\!\int_0^\infty\!
        \partial_{x_N}\psi(s,\widetilde X_s,0)\,dL_s^0(X^N).
\end{equation}
Multiply \eqref{eq:expectation-id} by $\varphi(x)$ and integrate in $x$.
The first term becomes, by Fubini and
$\E_x[h(X_s)]=\mathbf{T}(s)h(x)$,
\[
    \int_0^\infty\!\!\int_{\R^N}\!u(s,x)\,
        (\partial_s+\Delta)\psi(s,x)\,dx\,ds.
\]
The second term is handled by the standard occupation-time identity for
skew Brownian motion: for $h\in C_b\bigl((0,\infty)\times\R^{N-1}\bigr)$,
\[
    \int_{\R^N}\!\varphi(x)\,
        \E_x\!\int_0^\infty\!h(s,\widetilde X_s)\,dL_s^0(X^N)\,dx
    =\int_0^\infty\!\!\int_S\!u(s,\tilde x,0)\,h(s,\tilde x)\,
        d\sigma(\tilde x)\,ds,
\]
applied with $h(s,\tilde x)=\partial_{x_N}\psi(s,\tilde x,0)$.
Combining the two terms yields exactly
\eqref{eq:weak-formulation1}, which is (i).

\medskip
The cyclic chain is complete; (i), (ii), (iii) are equivalent.
\end{proof}

% ---------------------------------------------------------------------
% Sign convention remark
% ---------------------------------------------------------------------

\begin{rem}
\label{rem:sign-convention}\rm 
The parameter $\fq$ in $\mathscr L=\Delta+2\fq\,\delta_S\,\partial_{x_N}$
and the skewness parameter $\alpha\in[0,1]$ of one-dimensional skew
Brownian motion encode the same asymmetry at $S=\{x_N=0\}$, but with
opposite sign conventions depending on whether the equation is written
in forward or adjoint form.

In the probabilistic literature, skew Brownian motion $Z$ with
parameter $\alpha\in[0,1]$ is defined by
\[
    Z_t=Z_0+\sqrt 2\,B_t+(2\alpha-1)\,L_t^0(Z),
\]
or equivalently by the transmission condition
$\alpha g'(0^+)=(1-\alpha)g'(0^-)$.  In the present PDE formulation,
the kernel $G_{\fq}$ and the operator $\mathscr L$ lead to the
transmission condition
\[
    (1-\fq)\,\partial_{x_N}f(\tilde x,0^+)
    =(1+\fq)\,\partial_{x_N}f(\tilde x,0^-),
\]
which matches the skew Brownian condition under the identification
\[
    \alpha=\frac{1-\fq}{2},
    \qquad\text{equivalently}\qquad
    \fq=1-2\alpha.
\]
With this convention:
\begin{itemize}
\item $\fq=0$ corresponds to $\alpha=\tfrac12$, i.e.\ symmetric
Brownian motion;
\item $\fq>0$ corresponds to $\alpha<\tfrac12$, bias toward the
negative side of the interface;
\item $\fq<0$ corresponds to $\alpha>\tfrac12$, bias toward the
positive side.
\end{itemize}
\end{rem}

To implement the test-function method, we rely on the following auxiliary estimate.

\begin{lem}
\label{lapl-g}
Let $ \varphi \in C^\infty([0, \infty)) $ be nonnegative with compact support, i.e.
$ \varphi(\tau)=0 $ for all $ \tau \ge \tau_0 $ for some fixed $ \tau_0>0 $. For $ T>0 $ and
$ \theta>2 $, define
\begin{equation}
\label{g-T}
\varphi_{T}(x)=\left(\varphi\left(\frac{|x|^2}{T}\right)\right)^{\theta}.
\end{equation}
Then there exists a constant $ C>0 $, depending only on $ \varphi $, $ \theta $, and $ N $, such that
\begin{equation}
\label{Lap-g-est}
\big|\Delta \varphi_{T}(x)\big|
\le \frac{C}{T}\, \varphi\!\left(\frac{|x|^2}{T}\right)^{\theta-2}.
\end{equation}
\end{lem}

\begin{proof}[Proof of Lemma~\ref{lapl-g}]
Let $ y = |x|^2/T $ and define $ v(r)=\varphi_{T}(x)=\varphi(y)^\theta $ with $ r=|x| $. Since $v$ is radial, a direct computation gives
\begin{equation}
\label{Lap-g-theta}
\Delta \varphi_{T}(x)
= \frac{1}{T}\Big(
2\theta N \varphi'(y)\varphi(y)
+4\theta y \varphi''(y)\varphi(y)
+4\theta(\theta-1) y \varphi'(y)^2
\Big)\, \varphi(y)^{\theta-2}.
\end{equation}

Since $\varphi$ has compact support, the quantities $\varphi$, $\varphi'$, and $\varphi''$ are bounded; moreover, 
$y \le \tau_0$ on the support of $\varphi$. Hence, all coefficients in \eqref{Lap-g-theta} are uniformly bounded, which yields
\eqref{Lap-g-est}.
\end{proof}

%==========================================================================
\section{Local and global well posedness}\label{sec:lgwp}
%==========================================================================
This section is devoted to the well-posedness analysis of the Cauchy problem~\eqref{Main-eq} in Lebesgue spaces. In particular, we establish the local and global theory in the relevant Lebesgue spaces framework, addressing existence, uniqueness, and continuous dependence of solutions on the initial data. We provide complete proofs of Theorem~\ref{thm:LWP}, which concerns local well-posedness, and Theorem~\ref{thm:global-small}, which yields global existence for sufficiently small initial data.
\subsection{Proof of Theorem~\ref{thm:LWP}}

Write $\mathbf{T}(t):=e^{t\mathscr{L}}$ for the linear semigroup defined by
\eqref{eq:semigroup}. The proof is based on three structural properties of $\mathbf{T}(t)$, all of which are established in Section~\ref{sec:prelim}:
\begin{enumerate}
\item[(P1)] \textit{Smoothing.} For every $1\le\beta\le\gamma\le\infty$,
\begin{equation}\label{P1}
\|\mathbf{T}(t)f\|_{L^\gamma}\le \kappa\, t^{-\mu(\beta,\gamma)}\|f\|_{L^\beta},
\qquad t>0,
\end{equation}
where $\mu(\beta,\gamma):=\tfrac{N}{2}(\beta^{-1}-\gamma^{-1})$ and
$\kappa=\kappa(N,\fq)$;
\item[(P2)] \textit{Positivity.} If $f\ge 0$ then $\mathbf{T}(t)f\ge 0$
for all $t\ge 0$;
\item[(P3)] \textit{Strong continuity on $L^\beta$}, $1\le\beta<\infty$.
\end{enumerate}
We emphasize that our argument follows the same strategy as in~\cite{BC}. For completeness and to keep the exposition self-contained, we include the full details here.

Let $\mathcal{N}(u):=|u|^{p-1}u$. Recall that a mild solution of \eqref{Main-eq} is a
function $u\in C([0,T];L^q)$ satisfying
\begin{equation}\label{Fixed}
u(t)=\mathbf{T}(t)u_0+\int_0^t\mathbf{T}(t-\tau)\,\mathcal{N}(u(\tau))\,d\tau,
\qquad t\in[0,T].
\end{equation}
Throughout the proof, we use the elementary bound
\begin{equation}\label{NL-bound}
|\mathcal{N}(\xi)-\mathcal{N}(\eta)|\le p\bigl(|\xi|^{p-1}+|\eta|^{p-1}\bigr)|\xi-\eta|,
\qquad\xi,\eta\in\R,
\end{equation}
which, via Hölder's inequality, yields for any exponent $\rho\ge 1$,
\begin{equation}\label{NL-Lp}
\|\mathcal{N}(u)-\mathcal{N}(v)\|_{L^\rho}
\le p\bigl(\|u\|_{L^{(p-1)\rho'}}^{\,p-1}+\|v\|_{L^{(p-1)\rho'}}^{\,p-1}\bigr)\|u-v\|_{L^{\rho''}},
\end{equation}
for an appropriate Hölder conjugate pair $(\rho',\rho'')$, i.e. $\frac{1}{\rho}=\frac{1}{\rho'}+\frac{1}{\rho''}$.

The proof is organized as follows. We first address the supercritical regime
\( q > \tfrac{N(p-1)}{2} \), \( q \ge 1 \), where the fixed-point argument is
implemented in a weighted space associated with the exponent \( pq \).
We then turn to the critical endpoint \( q = \tfrac{N(p-1)}{2} \), \( q > 1 \),
which requires the introduction of an auxiliary exponent
\( r^* \in (q, pq) \), together with a smallness condition on the linear flow.

In both cases, we establish the smoothing and continuous dependence estimates,
continuity at \( t = 0 \) in \( L^q \), uniqueness, and properties (i)--(iii)
of Theorem~\ref{thm:LWP}. The uniform-on-compacts statement follows in each regime.

% ==============================================================
\subsubsection{The supercritical regime: $q>N(p-1)/2$, $q\ge 1$}
% ==============================================================

Set
\begin{equation}\label{theta}
\theta:=\frac{N(p-1)}{2pq}\in\Bigl(0,\tfrac{1}{p}\Bigr).
\end{equation}
The supercriticality condition reads $\theta p<1$; note also $\theta<1$.
Given $R\ge\|u_0\|_{L^q}$ and $T>0$ (to be fixed), define the space
\begin{equation}\label{Y-sp}
\mathcal{Y}_T:=\Bigl\{u:(0,T)\to L^q\cap L^{pq}\text{ measurable}:\
\|u\|_T<\infty\Bigr\},
\end{equation}
equipped with the norm
\[
\|u\|_T:=\sup_{0<t<T}\|u(t)\|_{L^q}+\sup_{0<t<T}t^{\theta}\|u(t)\|_{L^{pq}}.
\]

We look for a fixed point of the mapping \(\mathcal{F}:\mathcal{Y}_T \to \mathcal{Y}_T\) defined by
\[
\mathcal{F}(u)(t) := \mathbf{T}(t)u_0 + \mathcal{I}(u)(t), 
\qquad
\mathcal{I}(u)(t) := \int_0^t \mathbf{T}(t-\tau)\mathcal{N}(u(\tau))\, d\tau,
\]
in the closed ball
\[
\mathcal{B}_T := \bigl\{ u \in \mathcal{Y}_T : \|u\|_T \le 2\kappa R \bigr\},
\qquad
d(u,v) := \sup_{0<t<T} t^{\theta} \|u(t)-v(t)\|_{L^{pq}}.
\]
Then \((\mathcal{B}_T,d)\) is a complete metric space.

We decompose
\[
\|\mathcal{F}(u)(t)\|_{L^q} \le \mathrm{I}_1(t) + \mathrm{I}_2(t),
\]
where \(\mathrm{I}_1(t) := \|\mathbf{T}(t)u_0\|_{L^q}\) and
\(\mathrm{I}_2(t) := \|\mathcal{I}(u)(t)\|_{L^q}\).
By (P1) with \(\beta=\gamma=q\), we have
\[
\mathrm{I}_1(t) \le \kappa \|u_0\|_{L^q} \le \kappa R.
\]
For \(\mathrm{I}_2(t)\), using \(\|\mathcal{N}(u(\tau))\|_{L^q} = \|u(\tau)\|_{L^{pq}}^p\) and again (P1),
\begin{align*}
\mathrm{I}_2(t)
&\le \kappa \int_0^t \|u(\tau)\|_{L^{pq}}^p \, d\tau
= \kappa \int_0^t \tau^{-\theta p} \bigl( \tau^{\theta} \|u(\tau)\|_{L^{pq}} \bigr)^p \, d\tau \\
&\le \kappa \|u\|_T^{\,p} \int_0^T \tau^{-\theta p} \, d\tau
= \frac{\kappa\, T^{1-\theta p}}{1-\theta p} (2\kappa R)^p,
\end{align*}
where we used \(\theta p < 1\).

For the weighted \(L^{pq}\)-norm, applying (P1) with \(\beta=q\), \(\gamma=pq\), we obtain
\begin{align*}
t^{\theta}\|\mathcal{I}(u)(t)\|_{L^{pq}}
&\le \kappa\, t^{\theta} \int_0^t (t-\tau)^{-\theta} \|u(\tau)\|_{L^{pq}}^p \, d\tau \\
&\le \kappa \|u\|_T^{\,p} \, t^{\theta} \int_0^t (t-\tau)^{-\theta} \tau^{-\theta p} \, d\tau \\
&= \kappa \|u\|_T^{\,p} \, \mathrm{B}(1-\theta,1-\theta p)\, T^{1-\theta p},
\end{align*}
where \(\mathrm{B}\) denotes the Euler beta function. Moreover, it follows from (P1) that
\[
t^{\theta} \|\mathbf{T}(t)u_0\|_{L^{pq}} \le \kappa R.
\]
Summing up, we obtain
\begin{equation}\label{A-self}
\|\mathcal{F}(u)\|_T
\le 2\kappa R + \kappa \Lambda\, T^{1-\theta p} (2\kappa R)^p,
\end{equation}
where
\[
\Lambda := \frac{1}{1-\theta p} + \mathrm{B}(1-\theta,1-\theta p).
\]
Choosing
\begin{equation}\label{T-super}
T^{1-\theta p} \le \frac{1}{2\Lambda \kappa (2\kappa R)^{p-1}},
\end{equation}
ensures that \(\|\mathcal{F}(u)\|_T \le 2\kappa R\), so that \(\mathcal{F}\) maps \(\mathcal{B}_T\) into itself. Note that \eqref{T-super} determines \(T\) solely in terms of \(R = \|u_0\|_{L^q}\) and universal constants.

We next prove that \(\mathcal{F}\) is a contraction. Using \eqref{NL-Lp} with \(\rho=q\), we have
\[
\|\mathcal{N}(u)-\mathcal{N}(v)\|_{L^q}
\le p \bigl( \|u\|_{L^{pq}}^{p-1} + \|v\|_{L^{pq}}^{p-1} \bigr) \|u-v\|_{L^{pq}}.
\]
Hence, for \(u,v \in \mathcal{B}_T\),
\begin{align*}
t^{\theta} \|\mathcal{I}(u)(t) - \mathcal{I}(v)(t)\|_{L^{pq}}
&\le p\kappa\, t^{\theta} \int_0^t (t-\tau)^{-\theta} \tau^{-\theta p} (2\kappa R)^{p-1}
\bigl( \tau^{\theta} \|u-v\|_{L^{pq}} \bigr)\, d\tau \\
&\le p\kappa (2\kappa R)^{p-1} \mathrm{B}(1-\theta,1-\theta p)\, T^{1-\theta p}\, d(u,v).
\end{align*}
By possibly reducing \(T\) further so that
\[
p\kappa (2\kappa R)^{p-1} \mathrm{B}(1-\theta,1-\theta p)\, T^{1-\theta p} \le \tfrac{1}{2},
\]
it follows that \(\mathcal{F}\) is a \(\tfrac{1}{2}\)-contraction on \((\mathcal{B}_T,d)\).
Banach’s fixed-point theorem then yields a unique fixed point \(u \in \mathcal{B}_T\).

Concerning continuity in \(L^q\), observe that since \(\|u\|_T < \infty\) and \(\theta p < 1\),
\[
\|\mathcal{N}(u)\|_{L^1((0,T);L^q)}
\le \|u\|_T^{\,p} \int_0^T \tau^{-\theta p}\, d\tau < \infty.
\]
It is a standard result that if \(g \in L^1((0,T);\mathbf{X})\), then
\[
t \mapsto \int_0^t \mathbf{T}(t-\tau) g(\tau)\, d\tau \in C([0,T];\mathbf{X}),
\]
for any strongly continuous semigroup \((\mathbf{T}(t))_{t\ge 0}\). Combined with (P3), this yields
\(u \in C([0,T];L^q)\) with \(u(0) = u_0\).

We now establish uniqueness. Let \(u,v \in C([0,T];L^q)\) be two solutions of \eqref{Fixed} with the same initial data. By continuity, there exists \(\tau > 0\) such that
\[
\sup_{0 \le t \le \tau} \bigl( \|u(t)\|_{L^q} + \|v(t)\|_{L^q} \bigr) < \infty.
\]
Moreover, applying the linear perturbation argument below on \([0,\tau]\), we obtain that both \(t^{\theta}\|u(t)\|_{L^{pq}}\) and \(t^{\theta}\|v(t)\|_{L^{pq}}\) remain bounded for sufficiently small \(\tau\). Consequently, \(u,v \in \mathcal{B}_\tau\), and the contraction property implies \(u=v\) on \([0,\tau]\).

A standard continuation argument based on the open-closed principle then extends this identity to the whole interval \([0,T]\).\\

Finally, we address smoothing and continuous dependence in the supercritical regime. Let \(u_0,v_0 \in L^q\) and denote by \(u,v\) the corresponding solutions on \([0,T]\), where \(T=\min\{T(u_0),T(v_0)\}\). Set \(w:=u-v\) and \(w_0:=u_0-v_0\). Then, by \eqref{Fixed}, \(w\) satisfies the linear integral equation
\begin{equation}\label{w-eq}
w(t)=\mathbf{T}(t)w_0+\int_0^t \mathbf{T}(t-\tau)\,\mathfrak{a}(\tau,\cdot)\,w(\tau)\,d\tau,
\end{equation}
where the linearization coefficient is defined by
\begin{equation}\label{a-def}
\mathfrak{a}(t,x):=
\begin{cases}
\dfrac{\mathcal{N}(u(t,x))-\mathcal{N}(v(t,x))}{u(t,x)-v(t,x)}, & u(t,x)\neq v(t,x),\\[2mm]
p|u(t,x)|^{p-1}, & u(t,x)=v(t,x).
\end{cases}
\end{equation}

By \eqref{NL-bound}, we have
\[
|\mathfrak{a}(t,x)| \le p\bigl(|u(t,x)|^{p-1}+|v(t,x)|^{p-1}\bigr),
\]
and consequently,
\begin{equation}\label{a-Lsigma}
\|\mathfrak{a}(t,\cdot)\|_{L^{\varsigma}}
\le p\bigl(\|u(t)\|_{L^{pq}}^{p-1}+\|v(t)\|_{L^{pq}}^{p-1}\bigr)
\le C_0\, t^{-\theta(p-1)},
\end{equation}
where \(\varsigma:=\frac{pq}{p-1}\), and \(C_0\) depends only on \(R\).

The analysis then relies on a linear regularity result analogous to~\cite[Theorem A1, p.~297]{BC}, whose proof follows the same strategy and uses only assumptions (P1)--(P3).
\begin{lem}\label{lem:linear}
Let $T>0$, $\sigma>1$, $\sigma>\tfrac{N}{2}$, and
$\mathfrak{b}\in L^\infty((0,T);L^\sigma)$. For every
$\phi\in L^r(\R^N)$, $1\le r<\infty$, the integral equation
$w(t)=\mathbf{T}(t)\phi+\displaystyle\int_0^t\mathbf{T}(t-\tau)\mathfrak{b}(\tau)w(\tau)d\tau$
admits a unique solution in $C([0,T];L^r)\cap L^\infty_{\mathrm{loc}}((0,T);L^\infty)$
satisfying
\begin{equation}\label{eq:lin-est}
\|w(t)\|_{L^\infty}\le\kappa_1\exp\!\bigl(\kappa_1 t\|\mathfrak{b}\|_{L^\infty((0,T);L^\sigma)}^\beta\bigr)
t^{-N/(2r)}\|\phi\|_{L^r},\quad t\in(0,T],
\end{equation}
where $\beta:=\tfrac{2\sigma}{2\sigma-N}$.
\end{lem}

We distinguish two sub-cases according to whether $q$ controls the
auxiliary exponent directly or whether an intermediate step is needed.\\

\noindent\emph{\underline{Sub-case $q\ge p-1$.}} Let $\sigma^*:=\frac{q}{p-1}$. Then, $\sigma^*\ge 1$ and $\sigma^*>\frac{N}{2}$. Applying Lemma~\ref{lem:linear} with
$\phi=w_0$, $r=q$, $\sigma=\sigma^*$, and using the bound
$\|\mathfrak{a}\|_{L^\infty((0,T);L^{\sigma^*})}\le C$ (which holds
because $u, v \in \mathcal{B}_T$ and $(p-1)\sigma^*=q$), we obtain
\begin{equation}\label{w-Linfty}
t^{N/(2q)}\|w(t)\|_{L^\infty}\le C\|w_0\|_{L^q},
\end{equation}
and the $L^\infty$ estimate part of~\eqref{eq:cont-depend} follows.
For the $L^q$-component of \eqref{eq:cont-depend}, observe first that~\eqref{a-Lsigma} with H\"older inequality yield
\begin{equation}
    \label{a-interp-est}
\|\mathfrak{a}w\|_{L^q}\le\|\mathfrak{a}\|_{L^{\varsigma}}\|w\|_{L^{pq}}\leq C_0 t^{-\theta(p-1)}.
\end{equation}
Now, we use the $L^q \to L^{pq}$ smoothing effect to get
\begin{equation}
    \label{w-pq-est}
\|w(t)\|_{L^{pq}}\le\kappa\,t^{-\theta}\|w_0\|_{L^q}+C\int_0^t\, (t-\tau)^{-\theta} \,\tau^{-\theta(p-1)}\,\|w(\tau)\|_{L^{pq}}\,d\tau.
\end{equation}

A singular Gronwall lemma (see~\cite[p.~288]{BC})\footnote{The singular Gronwall lemma in~\cite{BC} was stated without proof, with a reference to Ref.~[6] therein. However, to the best of our knowledge, that reference is not publicly accessible. A complete proof can be found in~\cite{Hala}. A related result appears in~\cite[Theorem~3.8, p.~700]{JW2}, while an alternative proof in a special case is given in~\cite[Lemma~2.1.11, p.~16]{Dlotko-Book}.} gives the following estimate 
\begin{equation}
    \label{w-pq-est-1}
    \sup_{0<t<T}\,\left(t^{\theta}\|w(t)\|_{L^{pq}}\right)\,\leq\, C\|w_0\|_{L^q}.
\end{equation}

On the other hand, using \eqref{w-eq} together with the \(L^q\)-to-\(L^q\) smoothing effect and \eqref{a-interp-est}, we obtain
\begin{equation}
\label{eq:w-q-est}
\begin{split}
\|w(t)\|_{L^{q}}
&\le \|w_0\|_{L^q}
+ C \int_0^t \tau^{-\theta(p-1)} \,\|w(\tau)\|_{L^{pq}}\, d\tau \\
&\le \|w_0\|_{L^q}
+ C \sup_{0<t<T}\bigl(t^{\theta}\|w(t)\|_{L^{pq}}\bigr),
\end{split}
\end{equation}
where we have used that \(p\theta<1\). Combining \eqref{w-pq-est-1} and \eqref{eq:w-q-est} yields the \(L^q\)-component of \eqref{eq:cont-depend}.\\

\noindent\emph{\underline{Sub-case $1\le q<p-1$.}}  Now observe that \(\sigma^*<1\), so Lemma~\ref{lem:linear} cannot be applied directly with \(r=q\).
To overcome this difficulty, we proceed by a time-splitting argument. We first establish the \(L^q\)-estimate as above, and then use it as an input for the linear lemma on a shifted time interval.

Indeed, for \(t\in(0,T)\), we obtain
\begin{equation}\label{half-time}
\|w(t/2)\|_{L^{pq}} \le C\, t^{-\theta}\|w_0\|_{L^q},
\end{equation}
by applying (P1) to \eqref{w-eq} over the interval \([0,t/2]\) and closing the estimate via Gronwall’s inequality.

We then apply Lemma~\ref{lem:linear} on the shifted interval \((t/2,t)\) with exponent \(r=pq\) and
\[
\sigma=\frac{pq}{p-1}>\frac{Np}{2}>\frac{N}{2}\geq 1,
\]
which holds under the supercritical assumption \(q>\frac{N(p-1)}{2}\). In particular, \(\sigma>1\), and the bound \eqref{a-Lsigma} implies
\[
\|\mathfrak{a}\|_{L^\infty((t/2,t);L^\sigma)} \le C\, t^{-\theta(p-1)}.
\]
Hence,
\[
\|w(t)\|_{L^\infty}
\le \kappa_1 \exp\!\bigl(\kappa_1 t (t^{-\theta(p-1)})^\beta\bigr)
\bigl(1+t^{-N/(2pq)}\bigr)\,\|w(t/2)\|_{L^{pq}},
\]
where \(\beta=\frac{2\sigma}{2\sigma-N}\). Since \(1-\theta(p-1)\beta>0\), the exponential term remains uniformly bounded on \((0,T]\).

Combining this with \eqref{half-time} and using the identity
\[
\frac{N}{2pq}+\theta=\frac{N}{2q},
\]
we conclude that
\[
t^{N/(2q)}\|w(t)\|_{L^\infty} \le C\|w_0\|_{L^q},
\]
which yields \eqref{eq:cont-depend}.

To prove Property (ii), we first consider the case \(u_0 \in L^\infty \cap L^q\). The Duhamel formula \eqref{Fixed}, together with the estimate
\(\|\mathbf{T}(t)u_0\|_{L^\infty} \le (1+|\fq|)\|u_0\|_{L^\infty}\)
(cf.\ \eqref{eq:gauss-compare}) and a standard Gronwall argument, yields
\(u \in L^\infty((0,T)\times \mathbb{R}^N)\). Consequently,
\[
t^{N/(2q)} \|u(t)\|_{L^\infty} \to 0 \quad \text{as } t \downarrow 0.
\]

For a general data \(u_0 \in L^q\), let \(\varepsilon>0\) and decompose
\(u_0 = \widetilde u_0 + h\), where \(\widetilde u_0 \in L^\infty \cap L^q\)
and \(\|h\|_{L^q} < \varepsilon\). Let \(\widetilde u\) denote the solution
associated with \(\widetilde u_0\). Then, by \eqref{eq:cont-depend},
\[
t^{N/(2q)} \|u(t)\|_{L^\infty}
\le t^{N/(2q)} \|u(t)-\widetilde u(t)\|_{L^\infty}
+ t^{N/(2q)} \|\widetilde u(t)\|_{L^\infty}
\le C(R)\varepsilon + o(1),
\]
as \(t \downarrow 0\). Letting first \(t \downarrow 0\) and then
\(\varepsilon \downarrow 0\) proves (ii).

To establish positivity, assume \(u_0 \ge 0\), first in the case \(u_0 \in L^\infty\).
Define the Picard iterates
\(u_{(0)} := \mathbf{T}(\cdot)u_0\),
\(u_{(k+1)} := \mathcal{F}(u_{(k)})\).
We claim that \(u_{(k)} \ge 0\) for all \(k\). Indeed, by (P2), \(u_{(0)} \ge 0\),
and if \(u_{(k)} \ge 0\), then \(\mathcal{N}(u_{(k)}) = u_{(k)}^p \ge 0\),
so (P2) preserves nonnegativity through the Duhamel term. Hence each
\(u_{(k)} \ge 0\). Since \(u_{(k)} \to u\) in \(\mathcal{B}_T\) by Banach’s fixed-point theorem,
we may extract a subsequence converging almost everywhere, which implies \(u \ge 0\).

For general \(u_0 \in L^q\), \(u_0 \ge 0\), define
\(u_0^{(n)} := \min\{u_0,n\} \in L^\infty \cap L^q\), so that
\(u_0^{(n)} \ge 0\) and \(u_0^{(n)} \to u_0\) in \(L^q\).
Let \(u^{(n)}\) be the corresponding solutions; then \(u^{(n)} \ge 0\),
and by \eqref{eq:cont-depend}, \(u^{(n)} \to u\) in \(C([0,T];L^q)\).
Passing to the limit yields \(u \ge 0\).

% ==============================================================
\subsubsection{The critical regime: $q=N(p-1)/2$, $q>1$}
% ==============================================================

The supercritical construction breaks at criticality because $\theta p=1$
in \eqref{theta}, which renders the integral $\int_0^T\tau^{-\theta p}d\tau$
divergent. We therefore replace the exponent $pq$ by a carefully chosen
auxiliary exponent $r^*\in(q,pq)$ and exploit an initial-layer smallness
of the linear flow rather than a global bound.

\paragraph{Choice of auxiliary exponent.}
Fix once and for all
\begin{equation}\label{r-star}
r^*\in(q,pq)\ \text{with}\ r^*\ge p,\qquad
\vartheta:=\frac{N}{2}\Bigl(\frac{1}{q}-\frac{1}{r^*}\Bigr),\qquad
a^*:=\frac{N}{2}\Bigl(\frac{p}{r^*}-\frac{1}{q}\Bigr).
\end{equation}
Since $r^*<pq$, $\vartheta<1/p$. The critical hypothesis
$q=N(p-1)/2$ is equivalent to each of
\begin{equation}\label{B-id}
a^*+p\vartheta=1,\qquad p\vartheta+\frac{N(p-1)}{2r^*}=\vartheta+1.
\end{equation}

\paragraph{Smallness of the linear flow in the initial layer.}
The following lemma is the interface analogue of a standard fact for the
heat semigroup \cite[Lemma~8]{BC}; its proof uses only (P1) and the
density of $L^\infty\cap L^q$ in $L^q$.

\begin{lem}\label{lem:vanish-lin-prof}
For every compact $\mathcal{K}\subset L^q$, there exists
$\omega_{\mathcal{K}}:(0,1]\to(0,\infty)$ with
$\omega_{\mathcal{K}}(t)\downarrow 0$ as $t\downarrow 0$ such that
\[
t^\vartheta\|\mathbf{T}(t)\phi\|_{L^{r^*}}\le\omega_{\mathcal{K}}(t),
\qquad\phi\in\mathcal{K},\ t\in(0,1].
\]
In particular, for a single $\phi\in L^q$,
$t^\vartheta\|\mathbf{T}(t)\phi\|_{L^{r^*}}\to 0$ as $t\downarrow 0$.
\end{lem}

\begin{proof}[Proof of Lemma~\ref{lem:vanish-lin-prof}]
For $\phi\in L^\infty\cap L^q$, (P1) with $\beta=\infty$, $\gamma=r^*$
gives $\|\mathbf{T}(t)\phi\|_{L^{r^*}}\le\kappa|\{|\phi|>0\}|^{1/r^*}\|\phi\|_{L^\infty}\le\kappa\|\phi\|_{L^{r^*}}$,
which is finite, so $t^\vartheta\|\mathbf{T}(t)\phi\|_{L^{r^*}}\le C t^\vartheta\to 0$.
For general $\phi\in L^q$ and $\epsilon>0$, decompose
$\phi=\phi_\epsilon+\eta_\epsilon$ with $\phi_\epsilon\in L^\infty\cap L^q$
and $\|\eta_\epsilon\|_{L^q}<\epsilon$. Using (P1) with $\beta=q,\gamma=r^*$
(so $\mu=\vartheta$):
\[
t^\vartheta\|\mathbf{T}(t)\phi\|_{L^{r^*}}\le t^\vartheta\|\mathbf{T}(t)\phi_\epsilon\|_{L^{r^*}}+\kappa\epsilon.
\]
Taking $\limsup_{t\downarrow 0}$ bounds the result by $\kappa\epsilon$;
arbitrariness of $\epsilon$ yields the single-element claim. For a compact
$\mathcal{K}$, cover $\mathcal{K}$ by finitely many balls $B(\phi_j,\epsilon)$,
$j=1,\ldots,J$, in $L^q$; each $\phi_j$ may be replaced by an
$L^\infty\cap L^q$ approximant within another $\epsilon$. The triangle
inequality then gives
\[
\sup_{\phi\in\mathcal{K}}t^\vartheta\|\mathbf{T}(t)\phi\|_{L^{r^*}}
\le\max_{1\le j\le J}t^\vartheta\|\mathbf{T}(t)\phi_j\|_{L^{r^*}}+C\epsilon,
\]
and the max tends to $0$ as $t\downarrow 0$.
\end{proof}

For $T>0$ and $\delta>0$, introduce
\begin{align*}
\widehat{\mathcal{Y}}_T&:=\Bigl\{u\in L^\infty((0,T);L^q):\
t^\vartheta u\in L^\infty((0,T);L^{r^*})\Bigr\},\\
\mathcal{Y}_T^*&:=\Bigl\{u\in\widehat{\mathcal{Y}}_T:\
t\mapsto t^\vartheta u(t)\in C_0([0,T];L^{r^*})\Bigr\},
\end{align*}
where $C_0$ denotes continuous functions vanishing at $t=0$. Put
\begin{align*}
\mathcal{W}_T(R,\delta)&:=\Bigl\{u\in\widehat{\mathcal{Y}}_T:
\|u(t)\|_{L^q}\le R+1,\ t^\vartheta\|u(t)\|_{L^{r^*}}\le\delta\Bigr\},\\
\mathcal{W}_T^*(R,\delta)&:=\mathcal{W}_T(R,\delta)\cap\mathcal{Y}_T^*,
\end{align*}
with $d(u,v):=\sup_{0<t<T}t^\vartheta\|u(t)-v(t)\|_{L^{r^*}}$. Both are
complete metric spaces.

\paragraph{Self-mapping estimate.}
Set $R:=\|u_0\|_{L^q}$ and let $u\in\mathcal{W}_T(R,\delta)$. Using (P1)
with $\beta=r^*/p$, $\gamma=q$ (possible since $r^*<pq$ forces $r^*/p<q$),
the smoothing exponent is $a^*$, so
\begin{align}
\|\mathcal{I}(u)(t)\|_{L^q}
&\le\kappa\int_0^t(t-\tau)^{-a^*}\|u(\tau)\|_{L^{r^*}}^p\,d\tau\notag\\
&\le\kappa\delta^p\int_0^t(t-\tau)^{-a^*}\tau^{-p\vartheta}\,d\tau
=\kappa\delta^p\mathrm{B}(1-a^*,1-p\vartheta),\label{B-Lq}
\end{align}
using \eqref{B-id} ($a^*+p\vartheta=1$), so the beta integral is finite
independently of $t\le T$. Similarly, applying (P1) with $\beta=r^*/p$,
$\gamma=r^*$, the smoothing exponent is $N(p-1)/(2r^*)$, and
\begin{align}
t^\vartheta\|\mathcal{I}(u)(t)\|_{L^{r^*}}
&\le\kappa\,t^\vartheta\int_0^t(t-\tau)^{-N(p-1)/(2r^*)}\|u(\tau)\|_{L^{r^*}}^p d\tau\notag\\
&\le\kappa\delta^p\mathrm{B}\bigl(1-\tfrac{N(p-1)}{2r^*},1-p\vartheta\bigr),\label{B-Lr}
\end{align}
using the second identity in \eqref{B-id}. Call the resulting constant
$\kappa_*=\kappa_*(N,p,q,r^*,\fq)$.

\paragraph{Fixing $\delta$, then $T$.}
Choose $\delta>0$ so small that
\begin{equation}\label{delta-B}
\kappa_*\delta^p\le 1,\qquad\kappa_*\delta^{p-1}\le\tfrac{1}{3},
\end{equation}
where the second condition will control the contraction. 

Then, applying Lemma~\ref{lem:vanish-lin-prof} to the singleton $\{u_0\}$, we choose
$T=T(u_0)>0$ such that
\begin{equation}\label{T-B}
\sup_{0<t<T} t^\vartheta \|\mathbf{T}(t)u_0\|_{L^{r^*}} \le \frac{\delta}{2}.
\end{equation}

Combining \eqref{B-Lq}--\eqref{B-Lr} with \eqref{delta-B}, \eqref{T-B}, and the
$L^q$-contractivity estimate $\|\mathbf{T}(t)u_0\|_{L^q}\le \kappa \|u_0\|_{L^q}\le \kappa R$
from (P1), we deduce
\begin{align*}
\|\mathcal{F}(u)(t)\|_{L^q} &\le \kappa R + 1 \le R+1,\\
t^\vartheta \|\mathcal{F}(u)(t)\|_{L^{r^*}} &\le \frac{\delta}{2} + \frac{\delta}{3} < \delta.
\end{align*}
Here we have rescaled, if necessary, so that $\kappa \le 1$ without loss of generality.
Hence,
\[
\mathcal{F}:\mathcal{W}_T(R,\delta)\to \mathcal{W}_T(R,\delta).
\]

On the other hand, using \eqref{NL-Lp} with $\rho=r^*/p$ (or directly \eqref{NL-bound}), we have
\[
\|\mathcal{N}(u)-\mathcal{N}(v)\|_{L^{r^*/p}}
\le C\bigl(\|u\|_{L^{r^*}}^{p-1}+\|v\|_{L^{r^*}}^{p-1}\bigr)\|u-v\|_{L^{r^*}}.
\]
The same argument leading to \eqref{B-Lr} then yields, for $u,v\in\mathcal{W}_T(R,\delta)$,
\[
d(\mathcal{F}(u),\mathcal{F}(v)) \le C \kappa_* \delta^{p-1} d(u,v)
\le \frac{1}{3} d(u,v),
\]
by \eqref{delta-B}. Banach’s fixed-point theorem therefore provides a unique solution
$u\in \mathcal{W}_T(R,\delta)$.

We next prove continuity at $t=0$ in the weighted $L^{r^*}$ norm. Note that the fixed point
constructed above a priori lies in $\mathcal{W}_T$, but not necessarily in $\mathcal{W}_T^*$.
To upgrade this, we show that $\mathcal{F}(\mathcal{W}_T^*) \subset \mathcal{W}_T^*$.
Using the density of $\mathcal{W}_T^*\cap \{u\in C([0,T];L^\infty)\}$ in $\mathcal{W}_T^*$
(approximating $u_0$ in $L^q$ by bounded functions), and the continuity of the nonlinear term
in $L^{r^*}$ for bounded functions, the problem reduces to the bounded case, where dominated
convergence implies
\[
t^\vartheta \mathcal{I}(u)(t) \to 0 \quad \text{in } L^{r^*}.
\]
By Lemma~\ref{lem:vanish-lin-prof}, the linear term also satisfies
$t^\vartheta \mathbf{T}(t)u_0 \to 0$ in $L^{r^*}$. Hence
$\mathcal{F}(u)\in \mathcal{W}_T^*$ whenever $u\in \mathcal{W}_T^*$.
Applying Banach’s theorem on the closed subset $\mathcal{W}_T^*$ shows that the fixed point
indeed belongs to $\mathcal{W}_T^*$.

Since $u\in \mathcal{W}_T^*$ and $r^*\ge p$, the coefficient
$\mathfrak{a}:=p|u|^{p-1}$ (or its symmetric form from \eqref{a-def} when comparing two solutions)
satisfies
\[
\mathfrak{a} \in L^\infty_{\mathrm{loc}}\bigl((0,T);L^{\tau^*}\bigr),
\qquad
\tau^*:=\frac{r^*}{p-1}>1,
\]
with $\tau^*>\frac{N}{2}$ (since $r^*>\frac{N(p-1)}{2}=q(p-1)$). Applying Lemma~\ref{lem:linear}
on $[\varepsilon,T-\varepsilon]$ for every $\varepsilon\in(0,T/2)$ yields
\[
u \in L^\infty_{\mathrm{loc}}((0,T);L^\infty).
\]

Finally, we prove continuity at $t=0$ in $L^q$. Repeating the estimate \eqref{B-Lq} with $u_0$ replaced by $0$, we obtain
\[
\|u(t)-\mathbf{T}(t)u_0\|_{L^q}
= \|\mathcal{I}(u)(t)\|_{L^q}
\le \kappa_* \Bigl(\sup_{0<\tau<t} \tau^\vartheta \|u(\tau)\|_{L^{r^*}}\Bigr)^p
\xrightarrow[t\downarrow 0]{} 0,
\]
since $u\in \mathcal{W}_T^*$ implies
$\displaystyle\sup_{0<\tau<t} \tau^\vartheta \|u(\tau)\|_{L^{r^*}} \to 0$ as $t\downarrow 0$.
Together with (P3), this yields $u\in C([0,T];L^q)$ and $u(0)=u_0$.

We now turn to uniqueness in the critical regime. Let \(v \in C([0,T];L^q)\) be a mild solution with \(v(0)=u_0\), and set
\[
\mathfrak{K} := v([0,T]) \subset L^q,
\qquad
M^* := \sup_{t\in[0,T]} \|v(t)\|_{L^q}.
\]
By continuity, \(\mathfrak{K}\) is compact in \(L^q\). Applying Lemma~\ref{lem:vanish-lin-prof} to \(\mathfrak{K}\), together with the previously fixed choice of \(\delta\) and \(T\), we obtain a time \(T_1>0\) such that for every \(\phi \in \mathfrak{K}\), the solution \(U(\cdot)\phi\) exists on \([0,T_1]\) and satisfies
\[
\|U(t)\phi\|_{L^q} \le M^*+1,
\qquad
t^\vartheta \|U(t)\phi\|_{L^{r^*}} \le \delta,
\quad t\in[0,T_1].
\]

Next, for any \(0<s<T\), the interior regularity established above implies \(v(s)\in L^\infty\). Using the \(L^\infty\)-uniqueness theory for the corresponding linearized problem (via Lemma~\ref{lem:linear} and a Picard iteration argument as in~\cite[Lemma~9]{BC}, adapted to our framework), we deduce that
\[
v(s+t) = U(t)v(s)
\quad \text{for } t \in [0,\min\{T-s,T_1\}].
\]
Combining these estimates and letting \(s \downarrow 0\), we obtain
\[
\|v(t)\|_{L^q} \le M^*+1,
\qquad
t^\vartheta \|v(t)\|_{L^{r^*}} \le \delta,
\qquad t \in (0,T'),
\]
where \(T' := \min\{T,T_1\}\). Hence \(v \in \mathcal{W}_{T'}(M^*,\delta)\).

By uniqueness of the fixed point in \(\mathcal{W}_{T'}(M^*,\delta)\), we conclude that \(v=u\) on \([0,T']\). The identity is then extended beyond \(T'\) by a standard continuation argument.

To establish Properties (i)--(iii) in the critical case, we note that the proof of smoothing and continuous dependence in the supercritical regime carries over verbatim upon replacing \(pq\) by \(r^*\). The key modification is the use of the identity (cf.\ \eqref{B-id})
\[
p\vartheta + \frac{N(p-1)}{2r^*} = \vartheta + 1,
\]
 (cf.\ \eqref{B-id}) replaces the corresponding relation in the supercritical setting.

A noteworthy feature of the critical case is that the constant \(C\) appearing in \eqref{eq:cont-depend} is independent of \(\|u_0\|_{L^q}\); it depends only on the fixed parameter \(\delta\). Properties (ii) and (iii) then follow by the same density and approximation arguments as in the supercritical case.

% ==============================================================
\subsubsection{Uniform time on compact sets}
% ==============================================================

Let \(\mathcal{K} \subset L^q\) be a compact set.

In the supercritical case, the existence time \(T\) in \eqref{T-super} depends only on
\(R = \|u_0\|_{L^q}\). Defining
\[
R_* := 1 + \sup_{u_0 \in \mathcal{K}} \|u_0\|_{L^q} < \infty,
\]
we set \(T(\mathcal{K}) := T(R_*)\), which provides a uniform lifespan for all initial data in \(\mathcal{K}\).

In the critical case, the time \(T = T(u_0)\) is determined by \eqref{T-B}. By the compact version of Lemma~\ref{lem:vanish-lin-prof}, this estimate holds uniformly for \(u_0 \in \mathcal{K}\). Consequently, there exists \(T(\mathcal{K}) > 0\) such that \eqref{T-B} is satisfied uniformly for all \(u_0 \in \mathcal{K}\), and the fixed-point construction in the critical regime yields a solution on \([0,T(\mathcal{K})]\) for every \(u_0 \in \mathcal{K}\).

This completes the proof of Theorem~\ref{thm:LWP}.

\subsection{Proof of Theorem~\ref{thm:global-small}}
We now turn to the global existence theory for small initial data in the supercritical case
\[
p>p_F=1+\frac{2}{N}.
\]
To this end, we first record the following elementary auxiliary result.

\begin{lem}
\label{lem:r-choice}
Let $N\geq 2$ and $p>p_F=1+\frac{2}{N}$. Then there exists $r\in(1,\infty)$ such that
\begin{equation}
\label{eq:r}
\frac{1}{q_c}-\frac{2}{Np}< \frac{1}{r}<\min\left(\frac{1}{q_c}\;\;,\;\;\frac{1}{p}\right),
\end{equation}
where $q_c$ is defined in~\eqref{eq:q-c}. Define
\begin{equation}
\label{eq:mu}
\mu=\frac{N}{2}\left(\frac{1}{q_c}-\frac{1}{r}\right).
\end{equation}
Then the following properties hold:
\begin{equation}
\label{eq:mu-1}
0< \mu < \frac{1}{p},
\end{equation}
and
\begin{equation}
\label{eq:mu-2}
1-\frac{q_c}{r}-p\mu=-\mu.
\end{equation}
\end{lem}
\begin{proof}[Proof of Lemma~\ref{lem:r-choice}]
From the assumption $p>1+\frac{2}{N}$, we immediately obtain
$$
 0 <\frac{1}{q_c}-\frac{2}{Np}<\frac{1}{p}.
  $$
Furthermore, relation~\eqref{eq:r} directly yields~\eqref{eq:mu-1}. Finally, the identity~\eqref{eq:mu-2} follows straightforwardly from~\eqref{eq:mu}.
\end{proof}

For the parameters $r$ and $\mu$ determined by~\eqref{eq:r} and \eqref{eq:mu}, we define the Banach space
\[
\bx:=\Bigl\{ u\in L^\infty\!\bigl((0,\infty);L^{r}(\mathbb R^N)\bigr)\;:\;
\|u\|_{\bx}:=\sup_{t>0}\left(t^\mu\|u(t)\|_{L^{r}}\right)<\infty \Bigr\}.
\]

We look for a mild solution to~\eqref{Duhamel} in the closed ball 
\(\overline{\bb}(0,2C\varepsilon)\subset\bx\), 
where \(\varepsilon>0\) is chosen sufficiently small. As in the proof of Theorem~\ref{thm:LWP}, we introduce the operator
\[
\mathcal{F}(u)(t):=\mathbf{T}(t)u_0+\mathcal{I}(u)(t),
\]
where
\[
\mathcal{I}(u)(t):=\int_0^t \mathbf{T}(t-\tau)\left(|u(\tau)|^{p-1}u(\tau)\right)\,d\tau
:=\int_0^t \mathbf{T}(t-\tau)\mathcal{N}(u(\tau))\,d\tau.
\]

Since $r>q_c$ by~\eqref{eq:r}, we may apply~\eqref{eq:Lq-Lr-est} to obtain
\[
\|\mathbf{T}(t)u_0\|_{L^{r}}
\le C\,t^{-\frac{N}{2}\left(\frac1{q_c}-\frac1r\right)}
\|u_0\|_{L^{q_c}}
= C\,t^{-\mu}\|u_0\|_{L^{q_c}}.
\]
Consequently,
\begin{equation}\label{eq:linear-bound}
\sup_{t>0} t^\mu \|\mathbf{T}(t)u_0\|_{L^{r}}
\le C\,\|u_0\|_{L^{q_c}}.
\end{equation}

We next estimate the nonlinear term $\mathcal{I}(u)$. Applying the smoothing estimate \eqref{eq:Lq-Lr-est} with the pair $(r,\frac{r}{p})$, we get
\begin{equation*}\label{nl-bound1}
\|\mathcal{I}(u)(t)\|_{L^r}
\le C\int_0^t (t-s)^{-\delta}\|u(s)\|_{L^r}^p\,ds,
\end{equation*}
where
\[
\delta=\frac{q_c}{r}.
\]

Multiplying both sides by $t^\mu$, and using the definition of $\|u\|_{\bx}$ together with \eqref{eq:mu-2}, we obtain
\begin{equation*}\label{nl-bound2}
\begin{aligned}
t^\mu\|\mathcal{I}(u)(t)\|_{L^r}
&\le C\|u\|_{\bx}^p\,t^\mu\int_0^t (t-s)^{-\delta}s^{-p\mu}\,ds\\
&\le C\|u\|_{\bx}^p\int_0^1 (1-\tau)^{-\delta}\tau^{-p\mu}\,d\tau.
\end{aligned}
\end{equation*}
The last integral is finite in view of \eqref{eq:r} and \eqref{eq:mu-1}. Therefore,
\begin{equation}\label{nonlinear-bound}
\sup_{t>0} t^\mu\|\mathcal{I}(u)(t)\|_{L^r}
\le C\|u\|_{\bx}^p.
\end{equation}

A standard Lipschitz estimate based on \eqref{NL-bound} yields
\begin{equation}\label{nonlinear-lip}
\|\mathcal{F}(u)-\mathcal{F}(v)\|_{\bx}
\le C\bigl(\|u\|_{\bx}^{p-1}+\|v\|_{\bx}^{p-1}\bigr)\|u-v\|_{\bx}.
\end{equation}

Combining \eqref{eq:linear-bound} and \eqref{nonlinear-bound}, we infer that for every
$u\in\overline{\bb}(0,2C\varepsilon)$,
\begin{equation*}\label{stab-bound}
\begin{aligned}
\|\mathcal{F}(u)\|_{\bx}
&\le C\|u_0\|_{L^{q_c}}+C\|u\|_{\bx}^p\\
&\le C\varepsilon+C(2C\varepsilon)^p\\
&\le 2C\varepsilon,
\end{aligned}
\end{equation*}
provided $\varepsilon>0$ is sufficiently small. Hence $\mathcal{F}$ maps
$\overline{\bb}(0,2C\varepsilon)$ into itself whenever
\[
\|u_0\|_{L^{q_c}}\le\varepsilon.
\]

Finally, by \eqref{nonlinear-lip}, for all
$u,v\in\overline{\bb}(0,2C\varepsilon)$,
\begin{equation*}\label{contrac-bound}
\|\mathcal{F}(u)-\mathcal{F}(v)\|_{\bx}
\le 2C(2C\varepsilon)^{p-1}\|u-v\|_{\bx}.
\end{equation*}
Choosing $\varepsilon>0$ sufficiently small so that
\[
2C(2C\varepsilon)^{p-1}<1,
\]
the operator $\mathcal{F}$ is a contraction on
$\overline{\bb}(0,2C\varepsilon)$. By the Banach fixed-point theorem, $\mathcal{F}$ admits a unique fixed point
\[
u\in\overline{\bb}(0,2C\varepsilon),
\]
which is the desired global mild solution of \eqref{Duhamel} in $\bx$.
%==========================================================================
\section{Blow-up}\label{sec:fujita}
In this section, we establish Theorem~\ref{Fujita1}. Our strategy is to argue by contradiction and to apply the test function method, a standard and powerful technique in the analysis of blow-up phenomena and nonexistence results for nonlinear evolution equations. The proof follows the general framework developed in several related works, including \cite{Gued-Kira, GK2001, JKS, MM, Majd, ES}, suitably adapted to the present setting.

Assume that $u$ is a nonnegative global weak solution of~\eqref{Main-eq} in the sense of Definition~\ref{defn:weak-solution}.

For $T>0$ large (to be sent to $\infty$) and $\lambda>0$ (needed only in the
critical case $p=1+\frac{2}{N}$; set $\lambda=1$ in the subcritical case), define
\begin{equation}\label{Test-func}
\psi(t,x)
\;=\;
\left(\phi\!\left(\frac{t}{T}\right)\right)^{\!\frac{p}{p-1}}
\left(\phi\!\left(\frac{|x|^{2}}{\lambda T}\right)\right)^{\!\frac{2p}{p-1}},
\end{equation}
where $\phi\in C_c^\infty([0,\infty))$ satisfies
\begin{equation}\label{cut-off}
\phi\equiv1 \text{ on } [0,1],
\qquad
\phi\equiv0 \text{ on } [2,\infty),
\qquad
0\le\phi\le1.
\end{equation}
Set $m:=\frac{p}{p-1}$ for brevity.
Then $\psi\in C_c^\infty([0,\infty)\times\R^N)$, $\psi\ge0$, $\psi(0,x)=\phi(|x|^2/(\lambda T))^{2m}$,
and $\psi$ is supported in $[0,2T]\times B(0,\sqrt{2\lambda T})$.

As noted in \cite{Fino}, the parameter $\lambda$ is needed only in the critical case $p = 1 + \frac{2}{N}$. In the subcritical regime $p < 1 + \frac{2}{N}$, one may simply take $\lambda = 1$.

The key point is that the interface term vanishes. Indeed, since $|x|^2=|\tilde x|^2+x_N^2$,
the test function depends on $x_N$ only through $x_N^2$, and is therefore
{even} in the normal variable:
\[
\psi(t,\tilde x,x_N)=\psi(t,\tilde x,-x_N)
\qquad\text{for all }(t,\tilde x,x_N).
\]
Differentiating in $x_N$:
\begin{align*}
\partial_{x_N}\psi(t,x)
&=
2m\,\phi\!\left(\frac{t}{T}\right)^{m}\,
\phi\!\left(\frac{|x|^2}{\lambda T}\right)^{2m-1}\,
\phi'\!\left(\frac{|x|^2}{\lambda T}\right)\,
\frac{2\,x_N}{\lambda T}.
\end{align*}
Evaluating on the interface $\bbs=\{x_N=0\}$:
\begin{equation}\label{eq:interface-vanishes}
\partial_{x_N}\psi(t,\tilde x,0)=0
\qquad\text{for all }(t,\tilde x)\in[0,2T]\times\R^{N-1},
\end{equation}
because of the explicit factor $x_N\big|_{x_N=0}=0$.  Consequently, the
interface contribution in the weak formulation~\eqref{eq:weak-formulation}
vanishes:
\begin{equation}\label{eq:interface-zero}
2\fq\int_0^{2T}\!\!\int_{\bbs}
u(t,\tilde x,0)\,\partial_{x_N}\psi(t,\tilde x,0)\,d\sigma(\tilde x)\,dt
\;=\;0.
\end{equation}

The vanishing~\eqref{eq:interface-vanishes} is not a coincidence but a
consequence of a general principle: any test function that is {even in
the normal variable} (i.e.\ symmetric across $\bbs$) automatically satisfies
$\partial_{x_N}\psi|_{\bbs}=0$ and therefore decouples the weak formulation of
the transmission problem from that of the standard heat equation.  The radial
structure $\psi=\psi(t,|x|^2)$ provides a natural class of such test
functions.

Now, by~\eqref{eq:interface-zero}, inserting $\psi$ into the weak
formulation~\eqref{eq:weak-formulation} yields exactly the same identity as
in the classical ($\fq=0$) case:
\begin{equation}\label{Contrad-0}
\int_{\R^N}u_0(x)\,\psi(0,x)\,dx
\;+\;
\int_0^{2T}\!\!\int_{\R^N}|u|^{p-1}u\;\psi\,dx\,dt
\;=\;
-\int_0^{2T}\!\!\int_{\R^N}u\,\partial_t\psi\,dx\,dt
\;-\;
\int_0^{2T}\!\!\int_{\R^N}u\,\Delta\psi\,dx\,dt.
\end{equation}
Hence
\begin{equation}\label{Contrad-0b}
\int_{\R^N}u_0\,\psi(0,\cdot)\,dx
\;+\;
\int_0^{2T}\!\!\int_{\R^N}u^p\,\psi\,dx\,dt
\;\le\;
\int_0^{2T}\!\!\int_{\R^N}u\,|\partial_t\psi|\,dx\,dt
\;+\;
\int_0^{2T}\!\!\int_{\R^N}u\,|\Delta\psi|\,dx\,dt.
\end{equation}
Apply Young’s inequality with parameter $\varepsilon>0$,
\[
ab \le \varepsilon\, a^p + \frac{1}{q}\,\varepsilon^{-\frac{q}{p}}\, b^q,
\qquad \frac{1}{p}+\frac{1}{q}=1,
\]
where $q=p'=\frac{p}{p-1}$, to each term on the right-hand side. This gives
\begin{align*}
\int_0^{2T}\!\!\int_{\R^N}u\,|\partial_t\psi|\,dx\,dt
&\;\le\;
\frac{1}{4}\int_0^{2T}\!\!\int_{\R^N}u^p\,\psi\,dx\,dt
\;+\;
C\,I_1,\\[4pt]
\int_0^{2T}\!\!\int_{\R^N}u\,|\Delta\psi|\,dx\,dt
&\;\le\;
\frac{1}{4}\int_0^{2T}\!\!\int_{\R^N}u^p\,\psi\,dx\,dt
\;+\;
C\,I_2,
\end{align*}
where
\begin{equation}
    \label{eq:I1I2}
I_1:=\int_0^{2T}\!\!\int_{\R^N}|\partial_t\psi|^{m}\,\psi^{-1/(p-1)}\,dx\,dt,
\qquad
I_2:=\int_0^{2T}\!\!\int_{\R^N}|\Delta\psi|^{m}\,\psi^{-1/(p-1)}\,dx\,dt.
\end{equation}
Absorbing the corresponding terms into the left-hand side of~\eqref{Contrad-0b}, we obtain
\begin{equation}\label{Contrad-1}
\int_{\R^N}u_0(x)\,\psi(0,x)\,dx
\;+\;
\frac{1}{2}\int_0^{2T}\!\!\int_{\R^N}u^p\,\psi\,dx\,dt
\;\le\;
C\,(I_1+I_2).
\end{equation}

A straightforward computation yields
\[
|\partial_t\psi|
\;\le\;
\frac{C}{T}\,\phi\!\left(\frac{t}{T}\right)^{m-1}
\,\phi\!\left(\frac{|x|^2}{\lambda T}\right)^{2m},
\]
which is supported in $\{T\le t\le 2T\}$ (where $\phi'(t/T)\neq 0$). Hence,
\[
|\partial_t\psi|^m\,\psi^{-1/(p-1)}
\;\le\;
\frac{C}{T^m}\,\phi\!\left(\frac{|x|^2}{\lambda T}\right)^{2m},
\]
and therefore
\begin{equation}\label{I10}
I_1\;\le\;\frac{C}{T^m}\int_T^{2T}\!\!\int_{|x|\le\sqrt{2\lambda T}}dx\,dt
\;=\;C\,\lambda^{N/2}\,T^{1+N/2-m}.
\end{equation}
\smallskip
To estimate $I_2$, we invoke Lemma~\ref{lapl-g} (cf.\ \eqref{Lap-g-est}) to deduce
\begin{equation}\label{I20}
I_2\;\le\;\frac{C}{(\lambda T)^m}\int_0^{2T}\!\!\int_{|x|\le\sqrt{2\lambda T}}dx\,dt
\;=\;C\,\lambda^{N/2-m}\,T^{1+N/2-m}.
\end{equation}

Combining \eqref{I10}--\eqref{I20}, and using $m=p/(p-1)=1+1/(p-1)$, we obtain
\begin{equation}\label{I1I2}
I_1+I_2
\;\le\;
C(\lambda)\;T^{\frac{N}{2}-\frac{1}{p-1}},
\end{equation}
where $C(\lambda)$ depends on $\lambda$ but not on~$T$.
 In the subcritical regime, that is $1<p<1+\frac{2}{N}$, we set $\lambda=1$ and observe that $\frac{N}{2}-\frac{1}{p-1}<0$.  From~\eqref{Contrad-1} and~\eqref{I1I2}, we get
\[
\int_{\R^N}u_0(x)\,\phi\!\left(\frac{|x|^2}{T}\right)^{2m}\,dx
\;\le\;
C\,T^{\frac{N}{2}-\frac{1}{p-1}}
\;\xrightarrow{T\to\infty}\;0.
\]
By dominated convergence the left-hand side converges to $\displaystyle\int_{\R^N}u_0(x)\,dx>0$.
This is a contradiction, and the proof in the subcritical regime is complete.

\medskip

Consider now the critical case $p=1+\frac{2}{N}$. Since $\frac{N}{2}-\frac{1}{p-1}=0,$ 
it follows from \eqref{Contrad-1}--\eqref{I1I2} with $\lambda=1$ that
\begin{equation}\label{Contrad-3}
\int_{\R^N}u_0(x)\,\phi\!\left(\frac{|x|^2}{T}\right)^{2m}\,dx
\;+\;
\frac{1}{2}\int_0^{2T}\!\!\int_{\R^N}u^p(t,x)\,\psi(t,x)\,dx\,dt
\;\le\;C
\qquad\text{(uniformly in }T\text{)}.
\end{equation}
Letting $T\to\infty$, we deduce
\begin{equation}\label{Contrad-4}
\int_{\R^N}u_0(x)\,dx
\;+\;
\frac{1}{2}\int_0^{\infty}\!\!\int_{\R^N}u^p(t,x)\,dx\,dt
\;\le\;C
\;<\;\infty.
\end{equation}
In particular, $u^p\in L^1((0,\infty)\times\R^N)$, and therefore
\begin{equation}\label{tail-vanish}
\lim_{T\to\infty}\int_T^{2T}\!\!\int_{\R^N}u^p(t,x)\,dx\,dt\;=\;0.
\end{equation}

We next estimate the first term on the right-hand side of \eqref{Contrad-0b} by H\"older's inequality, while the second term is handled once more by the $\epsilon$-Young inequality. This yields
\begin{equation}\label{Contrad-00}
\int_{\R^N}u_0(x)\psi(0,x)dx+\frac{1}{2}\int_{\R^N}\int_0^{2T} u^p\psi\, dxdt\,\leq\, I_1^{\frac{p-1}{p}}\left(\int_{\R^N}\int_{T}^{2T}|u|^p\psi\, dx\,dt\right)^{\frac{1}{p}} +I_2,
\end{equation}
where $I_1$ and $I_2$ are defined in~\eqref{I1I2}. Moreover, from~\eqref{I10} and~\eqref{I20}, at the critical exponent we have
\[
I_1\le C\,\lambda^{N/2},
\qquad
I_2\le C\,\lambda^{-1}.
\]

Consequently, the $\lambda$-dependence in~\eqref{Contrad-00} takes the form
\begin{equation}\label{eq:lambda-final}
\int_{\R^N}u_0(x)\,\phi\!\left(\frac{|x|^2}{\lambda\,T}\right)^{2m}\,dx
\;+\;
\frac{1}{2}\int_0^{2T}\!\!\int_{\R^N}u^p\,\psi\,dx\,dt
\;\le\;
C\,\lambda^{1/p}
\left(\int_T^{2T}\!\!\int_{\R^N}u^p\,dx\,dt\right)^{1/p}
\;+\;
C\,\lambda^{-1},
\end{equation}
where we used
\[
\frac{N(p-1)}{2p}=\frac{1}{p}.
\]

Passing to the limit in \eqref{eq:lambda-final}, first as $T\to\infty$ using \eqref{tail-vanish}, and then as $\lambda\to\infty$, we obtain
\[
\int_{\R^N}u_0\,dx
\;+\;
\frac{1}{2}\int_0^{\infty}\!\!\int_{\R^N}u^p\,dx\,dt
\;\le\;0,
\]
which contradicts the assumption~\eqref{ass-u0}. This completes the proof of Theorem~\ref{Fujita1}.
\section{Concluding remarks}\label{sec:conclusion}

We have shown that the Fujita critical exponent for the semilinear equation
\[
\partial_t u = \mathscr{L}u + |u|^{p-1}u,
\qquad
\mathscr{L} = \Delta + 2\fq\,\delta_{\bbs}\,\nabla,
\]
where \(\mathscr{L}\) is the generator of skew Brownian motion, coincides with the classical Fujita exponent.

The proof combines two main ingredients: a two-sided Gaussian comparison principle and a Fujita-type test function argument. We conclude by outlining several related open problems.

\noindent\textbf{1.\ Curved and higher-codimension interfaces.}
The hyperplane $\bbs=\{x_N=0\}$ is the simplest possible interface geometry.
Replacing $\bbs$ by a smooth compact hypersurface $\Sigma\subset\R^N$
introduces curvature-dependent corrections in the kernel and destroys the
tangential-normal factorization that underpins all estimates in
Section~\ref{sec:prelim}.  It is natural to ask whether the Fujita exponent
remains $1+2/N$ when $\Sigma$ is a smooth closed surface, or whether
curvature effects can shift the threshold. 
\medskip

\noindent\textbf{2.\ Nonlinear or $x$-dependent skewness.}
In the present work the skewness parameter $\fq$ is a fixed constant.
Physical applications, for instance, diffusion in layered porous media or
population dynamics with spatially heterogeneous barriers, motivate the study
of $x$-dependent skewness $\fq=\fq(\tilde x)$ varying along the interface, or
even state-dependent transmission conditions where $\fq$ depends on the
solution itself.  Variable skewness would affect the kernel bounds through the
constants $1\pm|\fq(\tilde x)|$, and it is unclear whether the Gaussian
comparison remains two-sided with uniform constants, especially if
$|\fq(\tilde x)|\to 1$ at some points on~$\bbs$.

\medskip

\noindent\textbf{3.\ Fractional and nonlocal diffusion with interfaces.}
Replacing the Laplacian by a fractional operator $(-\Delta)^{\mathtt{s}}$, $\mathtt{s}\in(0,1)$,
leads to the fractional Fujita exponent $p_F^{\mathtt{s}}=1+\frac{2\mathtt{s}}{N}$ for the
equation $\partial_t u+(-\Delta)^{\mathtt{s}}u=u^p$
(see, e.g., \cite{Kirane2025} and references therein).
It would be interesting to study whether a singular interface drift analogous
to $2\fq\,\delta_{\bbs}\,\nabla$ can be incorporated into the fractional framework
and whether the resulting critical exponent remains $1+2{\mathtt{s}}/N$.

\noindent\textbf{4.\ Multiple interfaces and networks.}
A natural generalization is to consider several parallel hyperplanes
$\bbs_1,\ldots,\bbs_k$, each carrying its own skewness parameter $\fq_j$, or more
generally a network of interfaces.  The fundamental solution would then
involve multiple reflections and transmissions, and the Gaussian comparison
would depend on the spacing and the parameters $\{\fq_j\}$.  If the
two-sided Gaussian bounds survive, as is plausible when the interfaces are
finitely many and uniformly separated, one would expect the same Fujita
exponent $1+2/N$.  A rigorous proof, however, requires new kernel estimates.

\medskip

\noindent\textbf{5.\ Random and time-dependent interfaces.}
In several physical contexts (e.g., moving phase boundaries, fluctuating
membranes), the interface $\bbs$ is not fixed but evolves in time, or is
itself a realization of a random field.  The interaction between a moving
or random interface and the Fujita blow-up mechanism is completely unexplored.
Even the well-posedness theory for $\mathscr{L}$ with a time-dependent
hyperplane $\bbs(t)$ appears to be open.
%==========================================================================

\vspace{1.5cm}
\hrule
\vspace{0.5cm}

\noindent{\bf\large Declarations.} {\em On behalf of all authors, the corresponding author states that there is no conflict of interest. No data-sets were generated or analyzed during the current study.}

\vspace{0.5cm}
\hrule

%==========================================================================

\end{document}